\documentclass[letterpaper,11pt]{article}

\usepackage{macros}
\usepackage[margin=1in]{geometry}  %
\usepackage{xspace,exscale,relsize}
\usepackage{fancybox,shadow}
\usepackage{graphicx}
\usepackage{color}
\usepackage{amsthm,amsfonts}
\usepackage{amssymb}
\usepackage{authblk}
\usepackage[title]{appendix}

\usepackage{extarrows}

\usepackage[noadjust]{cite}

\usepackage{amsmath}
\makeatletter
\let\over=\@@over \let\overwithdelims=\@@overwithdelims
\let\atop=\@@atop \let\atopwithdelims=\@@atopwithdelims
\let\above=\@@above \let\abovewithdelims=\@@abovewithdelims
\makeatother
\usepackage{rotating}

\usepackage{ifpdf}

\usepackage{subfigure}
\usepackage{psfrag}

\usepackage{prettyref}
\usepackage{enumitem}

\usepackage{tikz}
\usetikzlibrary{arrows}
\tikzstyle{int}=[draw, fill=blue!20, minimum size=2em]
\tikzstyle{dot}=[circle, draw, fill=blue!20, minimum size=2em]
\tikzstyle{init} = [pin edge={to-,thin,black}]

\usepackage[%
  CJKbookmarks=true,
  bookmarksnumbered=true,
  bookmarksopen=true,
  colorlinks=true,
  citecolor=red,
  linkcolor=blue,
  anchorcolor=red,
  urlcolor=blue
]{hyperref}

\usepackage[all]{xy}

\usepackage{mathtools}

\usepackage{ifthen}
\newboolean{aos}
\setboolean{aos}{FALSE}

\newcommand\numberthis{\addtocounter{equation}{1}\tag{\theequation}}

\ifx\eqref\undefined
\newcommand{\eqref}[1]{~(\ref{#1})}
\fi
\ifx\mod\undefined
\def\mod{\mathop{\rm mod}}
\fi

\usepackage{bm}

\newcommand{\norm}[1]{{\left\Vert #1 \right\Vert}}

\def\argmax{\mathop{\rm argmax}}

\def\exp{\mathop{\rm exp}}

\def\Var{\mathrm{Var}}

\def\simiid{\stackrel{iid}{\sim}}

\newcommand{\reals}{\mathbb{R}}

\newcommand{\Expect}{\mathbb{E}}

\newcommand{\Prob}{\mathbb{P}}
\newcommand{\prob}[1]{\mathbb{P}\left[#1\right]}

\newcommand{\TV}{{\rm TV}}

\newcommand{\iid}{iid\xspace}

\newcommand{\hellinger}{d_{\rm H}}

\newcommand{\pth}[1]{\left( #1 \right)}

\newcommand{\sth}[1]{\left\{ #1 \right\}}

\newcommand{\iiddistr}{\stackrel{\text{\iid}}{\sim}}

\newcommand{\var}{\Var}
\newcommand\indep{\protect\mathpalette{\protect\independenT}{\perp}}
\def\independenT#1#2{\mathrel{\rlap{$#1#2$}\mkern2mu{#1#2}}}
\newcommand{\Bern}{\text{Bern}}

\newcommand{\Iprod}[2]{\langle #1, #2 \rangle}
\newcommand{\indc}[1]{{1\{{#1}\}}}

\definecolor{myblue}{rgb}{.8, .8, 1}
\definecolor{mathblue}{rgb}{0.2472, 0.24, 0.6} %
\definecolor{mathred}{rgb}{0.6, 0.24, 0.442893}
\definecolor{mathyellow}{rgb}{0.6, 0.547014, 0.24}

\newcommand{\calA}{{\mathcal{A}}}

\newcommand{\calF}{{\mathcal{F}}}
\newcommand{\calG}{{\mathcal{G}}}

\newcommand{\calM}{{\mathcal{M}}}

\newcommand{\calP}{{\mathcal{P}}}

\newrefformat{eq}{(\ref{#1})}
\newrefformat{thm}{Theorem~\ref{#1}}
\newrefformat{th}{Theorem~\ref{#1}}
\newrefformat{chap}{Chapter~\ref{#1}}
\newrefformat{sec}{Section~\ref{#1}}
\newrefformat{seca}{Section~\ref{#1}}
\newrefformat{algo}{Algorithm~\ref{#1}}
\newrefformat{fig}{Fig.~\ref{#1}}
\newrefformat{tab}{Table~\ref{#1}}
\newrefformat{rmk}{Remark~\ref{#1}}
\newrefformat{clm}{Claim~\ref{#1}}
\newrefformat{def}{Definition~\ref{#1}}
\newrefformat{cor}{Corollary~\ref{#1}}
\newrefformat{lmm}{Lemma~\ref{#1}}
\newrefformat{prop}{Proposition~\ref{#1}}
\newrefformat{pr}{Proposition~\ref{#1}}
\newrefformat{app}{Appendix~\ref{#1}}
\newrefformat{apx}{Appendix~\ref{#1}}
\newrefformat{ex}{Example~\ref{#1}}
\newrefformat{exer}{Exercise~\ref{#1}}
\newrefformat{soln}{Solution~\ref{#1}}

\def\unifto{\mathop{{\mskip 3mu plus 2mu minus 1mu%
      \setbox0=\hbox{$\mathchar"3221$}%
      \raise.6ex\copy0\kern-\wd0%
\lower0.5ex\hbox{$\mathchar"3221$}}\mskip 3mu plus 2mu minus 1mu}}

\ifx\lesssim\undefined
\def\simleq{{{\mskip 3mu plus 2mu minus 1mu%
      \setbox0=\hbox{$\mathchar"013C$}%
      \raise.2ex\copy0\kern-\wd0%
\lower0.9ex\hbox{$\mathchar"0218$}}\mskip 3mu plus 2mu minus 1mu}}
\else
\def\simleq{\lesssim}
\fi

\ifx\gtrsim\undefined
\def\simgeq{{{\mskip 3mu plus 2mu minus 1mu%
      \setbox0=\hbox{$\mathchar"013E$}%
      \raise.2ex\copy0\kern-\wd0%
\lower0.9ex\hbox{$\mathchar"0218$}}\mskip 3mu plus 2mu minus 1mu}}
\else
\def\simgeq{\gtrsim}
\fi

    \newtheorem{theorem}{Theorem}
    \newtheorem{lemma}[theorem]{Lemma}
    \newtheorem{corollary}[theorem]{Corollary}

    \newtheorem{prop}[theorem]{Proposition}

    \theoremstyle{definition}

    \newtheorem{remark}{Remark}

    \newif\ifmapx
    {\catcode`/=0 \catcode`\\=12/gdef/mkillslash\#1{#1}}
    \edef\jobnametmp{\expandafter\string\csname embayes2_apx\endcsname}
    \edef\jobnameapx{\expandafter\mkillslash\jobnametmp}
    \edef\jobnameexpand{\jobname}
    \ifx\jobnameexpand\jobnameapx
    \mapxtrue
    \else
    \mapxfalse
    \fi

    \newcommand{\polylog}{\operatorname{polylog}}

\newcommand{\GNPMLE}{\hat G_{\rm NPMLE}}

    \newcommand{\Regret}{\mathsf{Regret}}

    \renewcommand{\hat}{\widehat}
    \renewcommand{\tilde}{\widetilde}

    \newcommand{\Law}{\text{Law}}
    \newcommand{\op}{\mathrm{op}}

\begin{document}
    \ifpdf
    \DeclareGraphicsExtensions{.pgf}
    \graphicspath{{figures/}{plots/}}
    \fi

    \title{Sharp regret--Hellinger bounds and optimal rates for Gaussian empirical
    Bayes via polynomial approximation}
    \author{Jiafeng Chen and Yihong Wu\thanks{
J.~Chen is with the Department of Economics, Stanford University, Stanford, CA, USA, \texttt{jiafeng@stanford.edu}.
Y.~Wu is with the Department of Statistics and Data Science, Yale University, New Haven, CT, USA,
\texttt{yihong.wu@yale.edu}.
}}
\date{\today}

\maketitle

\begin{abstract}
A central problem in the theory of empirical Bayes is to control the regret (excess risk) of a
learned Bayes rule by the Hellinger distance between the estimated and true marginal
  densities. In the normal means model, the classical result of Jiang and Zhang 
  \cite{JiangZhang2009} achieves this only after regularizing the Bayes rule and incurs an extraneous cubic logarithmic factor through a delicate recursive argument.
  
  This paper introduces a new technique, based on polynomial approximation and
   Bernstein-type inequalities for weighted $L_2$ norms, that bounds the unregularized
   regret directly. The method is conceptually simpler and yields sharper, sometimes
   optimal, regret bounds. For compactly supported priors, we prove the sharp bound that
   the regret is  $O(\epsilon^2   \frac{\log(1/\epsilon)}{\log\log(1/\epsilon)})$,
   where $\epsilon$ is the Hellinger distance between the marginal densities. The same
   method also extends to priors with exponential tails. Conversely, we show that
   regularization is genuinely necessary for heavy-tailed priors under only bounded moment
   assumptions. 
	
	As statistical consequences, we obtain improved regret bounds for the
   nonparametric maximum likelihood estimator (NPMLE). 
	Notably, for compactly supported priors, 
	by determining the optimal Hellinger rate of mixture density estimation, we show that the optimal regret for sample size $n$ scales as $\Theta(\frac{1}{n}
(\frac{\log n}{\log\log n})^2)$, attained by the NPMLE within a $\log\log n$ factor.
\end{abstract}

\tableofcontents

\section{Introduction}

Empirical Bayes (EB) is a powerful framework for large-scale statistical inference that applies Bayesian methods with a data-driven prior, thereby adapting to the latent structure in the data \cite{robbins1956empirical,zhang2003compound}. For an account of the remarkable successes of EB in both theory and practice, we refer to the monograph and survey of Efron \cite{efron2012large,efron2024empirical}. A key benchmark for an EB procedure is its \textit{regret}: the excess risk relative to the Bayesian oracle that knows the true prior, or, in the compound setting, the empirical distribution of the parameters.

A standard setting for EB is the normal means model. 
Consider $\theta \sim G$ and $Y=\theta+Z$, where $Z\sim N(0,1)$ is independent of $\theta$. Tweedie's formula expresses the Bayes estimator in terms of the \textit{score} of the 
marginal distribution of $Y$, namely, the Gaussian mixture $f_G := G * N(0,1)$:
    \begin{equation}
      m_G(y) := \Expect[\theta|Y=y]  = y + \frac{f'_G(y)}{f_G(y)}.
      \label{eq:tweedie}
    \end{equation}
    This estimator minimizes the mean-squared error (MSE) among all estimators of $\theta$: 
$\min_{\hat\theta(\cdot)} \Expect_G[(\theta-\hat\theta(Y))^2] = 
\Expect_G[(\theta-m_G(Y))^2]$.

In the context of empirical Bayes, $m_G$ in \prettyref{eq:tweedie} is referred to as \textit{oracle Bayes}, which requires the knowledge of the true prior $G$. 
When 
the Bayes rule $m_H$ with  a misspecified prior $H$ is used, the regret, namely,  the excess 
MSE over the oracle risk, is given by
\begin{equation}
\Regret(H\|G) := 
\Expect_G[(m_H(Y)-m_G(Y))^2] 
= 
\int_{\reals} \pth{\frac{f_G'}{f_G}-\frac{f_H'}{f_H}}^2 f_G dy,
\label{eq:regret}
\end{equation}
where the second equality applies Tweedie's formula \prettyref{eq:tweedie}.
This quantity is also known as \textit{relative Fisher information} 
in information theory and optimal transport \cite{villani.topics,raginsky2013concentration} and  \textit{score matching loss} in generative modeling \cite{hyvarinen2005estimation}.

In practice---and especially in the 
\textit{$g$-modeling} approach  \cite{efron2014two}---$H$ is a prior estimated from a
 training sample and $\Regret(H\|G)$ is the incurred regret when the learned Bayes rule
 $m_H$ is applied to a fresh test data point.
Widely used approaches for learning the prior include maximum likelihood and minimum
distance, both of which aim at approximating the true data-generating distribution $f_G$.
As such, a central problem in EB theory is to control the regret $\Regret(H\|G)$ using
statistical distances between $f_G$ and $f_H$, most commonly the squared Hellinger
distance:
\begin{equation}
  \hellinger^2(f_G,f_H) :=
  \int (\sqrt{f_G}-\sqrt{f_H})^2.
  \label{eq:hellinger}
\end{equation}

A foundational result in nonparametric EB of this type is due to Jiang and Zhang
\cite{JiangZhang2009}.
A substantial body of subsequent work in statistics and econometrics builds on this result; see, for
instance, \cite{jiang2020general,saha2020nonparametric,soloff2025multivariate,SW22,adusumilli2025empirical,ghosh2025stein,chen2026empirical,ignatiadis2026compound,kim2026empirical,chen2026normal}. In addition, this result also finds applications beyond the classical realm of EB; for example, it plays a crucial role in constructing optimal kernel-based estimators of score functions for diffusion models \cite{WWY24}.

A key ingredient
of Jiang--Zhang's result \cite[Lemma 1]{JiangZhang2009}
controls the following regularized version of the regret \prettyref{eq:regret}:
\[
\Regret_{\rho}(H \Vert G) := \int_\R \pr{
  \frac{f'_G}{f_G \vee \rho} - \frac{f'_H}{ f_H \vee \rho}
}^2 f_G .
\numberthis\label{eq:regularized_regret}
\]
This quantity is related to the regret of the \textit{regularized} Bayes estimator
\begin{equation}
 m_{H,\rho}(y) := y + \frac{f'_H(y)}{f_H(y)\vee\rho}
    \label{eq:bayes-rho}
\end{equation}
where the regularization parameter $\rho$ ensures the denominator is not too small.
Using a delicate recursive argument, 
Jiang and Zhang showed that \prettyref{eq:regularized_regret} can be bounded by the Hellinger distance 
$\epsilon := \hellinger(f_G,f_H)$ as follows:
\begin{equation}
\Regret_{\rho}(H \Vert G)
\lesssim \epsilon^2 \max\sth{\pth{\log \frac{1}{\rho}}^3, \log \frac{1}{\epsilon}}.
    \label{eq:JZ09}
\end{equation}
This quantity can be further related to the regret of the regularized estimator \prettyref{eq:bayes-rho} by quantifying the regularization error on the true Bayes rule, under suitable tail assumptions on the prior $G$. 
In practice, $\rho$ is typically chosen as a small polynomial of $\epsilon$, leading to a regret bound of the order $O(\epsilon^2 (\log \frac{1}{\epsilon})^3)$.

The key difficulty of proving a regret--Hellinger inequality like \prettyref{eq:JZ09} is
that regret compares the scores (derivative of log density) while Hellinger only compares
the densities themselves. To bridge this gap, Jiang and Zhang introduced a sequence $\Delta_k$ defined in terms of the $k$th derivative of $f_G-f_H$, where  $\Delta_0$ and $\Delta_1$ are approximately the Hellinger and the regularized regret, respectively. 
They then proved a remarkable recursive inequality 
$\Delta_k^2 \leq \Delta_{k-1}(\Delta_{k+1} + C \sqrt{\log(1/\rho)}  \Delta_k)$ and bounded  $\Delta_k$ separately using Fourier analysis, concluding the desired \prettyref{eq:JZ09} by backward induction.

Since its introduction, the Jiang--Zhang method has remained essentially the only route for
bounding the regret of $g$-modeling estimators,\footnote{One exception is \cite{JPW21} for
the Poisson EB, where the regret can be bounded optimally by the Hellinger thanks to the
discreteness of the Poisson model and the light tail of the prior. Nevertheless, this
method fails for the moment class and an argument \`a la  Jiang--Zhang is still needed 
\cite{SW22}.} %
with extensions to multiple dimensions \cite{saha2020nonparametric}, higher-order estimands
\cite{ghosh2025stein} (see also an earlier preprint of \cite{chen2026empirical}),
heteroskedastic variances \cite{jiang2020general,soloff2025multivariate,chen2026empirical}, and the Poisson model 
\cite{SW22}.
At the same time, it has long been suspected that this argument is not
tight,\footnote{\cite{JiangZhang2009} remarked that ``Still, the cubic power of the
logarithmic factors [...]
is crude.'' and that
``we do not believe that mathematical
induction is sharp in the argument with higher and higher order of differentiation
in the proof of Lemma 1.''
} and this looseness contributes to a $\polylog(n)$
gap between the best known regret upper bound and information-theoretic limits
\cite{polyanskiy2021sharp,kang2026function}. Moreover, it is also not obvious whether the
regularization in \eqref{eq:regularized_regret} is merely a proof technique or plays a 
more fundamental role for empirical Bayes.

\paragraph{Main contributions.}
We present a new technique for bounding regret in terms of Hellinger distance
\textit{without regularization}. Compared to the recursive argument of Jiang and Zhang,
the proposed method is conceptually simpler and yields sharper regret bounds. For
compactly supported priors, these regret bounds are \emph{optimal}, matching the lower
bound of \cite{polyanskiy2021sharp}, after combining our regret-Hellinger results with improved
local metric and bracketing entropy calculations that yield sharper convergence rates in
Hellinger distance.

The major contributions of this paper are as follows.

\begin{itemize}
    \item For priors with compact support, we show that
\begin{equation}    
\Regret(H \Vert G) \lesssim \epsilon^2 \frac{\log \frac{1}{\epsilon}}{\log \log \frac{1}{\epsilon}},
\label{eq:regret-compact-intro}
\end{equation}
and prove that this is unimprovable in the worst case. 
This result has several statistical consequences, including improved regret guarantees for the nonparametric maximum likelihood estimator (NPMLE) in EB estimation.\footnote{For additional statistical applications of \prettyref{eq:regret-compact-intro}, see the recent work \cite{han2026empirical} on EB estimation with correlated data.}

Notably, when applied to the EB setting with sample size $n$, 
\prettyref{eq:regret-compact-intro} yields sharp rates for regret.
By sharpening local entropy bounds via 
similar polynomial approximation arguments, we determine the optimal squared Hellinger rate of density estimation for the mixture density to be $\Theta(\frac{1}{n}
\frac{\log n}{\log\log n})$, and show that any optimal density estimator attains the optimal regret of $\Theta(\frac{1}{n}
(\frac{\log n}{\log\log n})^2)$, meeting the minimax lower bound of 
\cite{polyanskiy2021sharp}. Additionally, the NPMLE (but
constrained to a compact support) is regret-optimal within $\log \log n$, compared to the
previously known best bound of $O(\frac{1}{n} (\log n)^5)$ \cite{JiangZhang2009}.

\item We extend this argument to priors with exponential tails, such as subgaussian or
subexponential priors. For subgaussian priors, we show that $\Regret(H\Vert G) \lesssim \epsilon^2 \log \frac{1}{\epsilon} (\log \log \frac{1}{\epsilon})^2$, which is optimal up to $\log\log$ factors.

\item Finally, attempting to extend this argument to heavier-tailed priors reveals that
the regularization in \cite{JiangZhang2009} is in fact necessary in this regime. We show
that, under only a bounded $p$th-moment assumption, the unregularized regret can be as
large as $\Omega (\epsilon^ {2-2/p})$, while the regularized regret bound \prettyref{eq:JZ09} continues to hold.
\end{itemize}

\paragraph{Proof ideas.}
As discussed earlier, controlling regret in terms of Hellinger distance entails bounding a derivative in terms of the original function.
Such a bound is clearly impossible in general, but it becomes possible for certain classes
of regular functions, most notably \textit{polynomials}. This crucial insight is
initially gained
through interactions with GPT-5.4 using OpenAI Codex. It motivates the following
program:
\begin{enumerate}
    \item 
    Identify a target function $g$ and a weight function $w$ such that 
    the Hellinger distance $\hellinger^2(f_G,f_H)$ is largely determined by 
    $\|g\|_{L_2(w)}^2 := \int g^2 w dy$ and the regret  $\Regret(H\|G)$  by 
    $\|g'\|_{L_2(w)}^2$.

    \item Show that the weight $w$ satisfies a \emph{Bernstein-type inequality}:
    for any degree-$k$ polynomial $p$,
    \begin{equation}
    \|p'\|_{L_2(w)} \leq a_k 
    \|p\|_{L_2(w)} 
        \label{eq:bernstein-general}
    \end{equation}

    \item Quantify the approximation error: find a degree-$k$ polynomial  $p$ so that $\|g-p\|_{L_2(w)}$ and 
    $\|g'-p'\|_{L_2(w)}$ are bounded in terms of $k$.
\end{enumerate}
Putting these ingredients together and optimizing the degree yields a regret--Hellinger inequality.

There are multiple ways to execute this program. To fix ideas, let us focus on the class of compactly supported priors. One may choose 
$g = \frac{f_G-f_H}{f_G+f_H}$ and $w=(f_G+f_H)/2 = \frac{G+H}{2}*\varphi$,
in order to capitalize on the classical Bernstein inequality for the standard Gaussian density $\varphi$ (an immediate consequence of properties of Hermite polynomials; see \prettyref{rmk:bernstein-perturb}):
    \begin{equation}
    \|p'\|_{L_2(\varphi)} \leq \sqrt{k} 
    \|p\|_{L_2(\varphi)} 
        \label{eq:bernstein-gaussian}
    \end{equation}
and hence (by linearity) for any Gaussian mixture weight. In addition,
the best degree-$k$ polynomial approximation error of this choice of $g$ is at most $\exp(-\Omega(\sqrt{k}))$. 
Choosing $k =\Theta((\log \frac{1}{\epsilon})^2)$ leads to $\Regret(H\|G) = O(\epsilon^2 (\log \frac{1}{\epsilon})^2)$, which already improves over the state of the art \cite{JiangZhang2009} and does so without regularization.

To obtain the sharp inequality, we instead choose 
\[
g = \frac{f_G-f_H}{\varphi}, \quad w=  \frac{\varphi^2}{(f_G+f_H)/2}.
\]
This $g$ admits a substantially better polynomial approximation error of $\exp(-\Omega(k
\log k))$. In addition, we prove a Bernstein-type inequality 
\prettyref{eq:bernstein-general} for the non-standard weight $w$ with $a_k = O(\sqrt{k})$.
Setting 
$k =\Theta(\log \frac{1}{\epsilon}/ \log\log \frac{1}{\epsilon})$
yields the desired \prettyref{eq:regret-compact-intro}.
We note that there is a sizable literature on Bernstein-type inequalities for weighted
norms (see, e.g., \cite{LevinLubinsky2001,lubinsky2007survey}); however these results
require additional assumptions that appear difficult to verify in our setting. Instead, we
present self-contained proofs using a linear algebraic argument, as the inequality \prettyref{eq:bernstein-general} boils down to bounding the operator norm of the matrix representing the differentiation operator on the space of polynomials.

\section{Main results}

\subsection{New Regret--Hellinger inequality}

\begin{theorem}
  \label{thm:ub-compact-intro}
  Let $G,H$ be supported on $[-M,M]$.
  There exists a constant $C=C(M)$ such that
  \[
    \Regret(H\|G)
    \leq
    C   \epsilon^2 \frac{\log \frac{1}{\epsilon}}{\log \log \frac{1}{\epsilon}}
  \]
  where $\epsilon^2 = \hellinger^2(f_{G},f_{H})$.
\end{theorem}

The following result shows that the regret upper bound in \prettyref{thm:ub-compact-intro} is sharp up to constants:
\begin{theorem}
      There exists a sequence $(G_m,H_m)$ all supported on $[-1,1]$, such that
      $\epsilon_m^2 \equiv \hellinger^2(f_{G_m},f_{H_m}) \to 0$ and
      \[
        \Regret(H_m\|G_m)
        \geq
        c_0 \epsilon_m^2 \frac{\log \frac{1}{\epsilon_m}}{\log \log \frac{1}{\epsilon_m}}
      \]
      where $c_0$ is an absolute constant.
      \label{thm:lb}
    \end{theorem}

We also extend \prettyref{thm:ub-compact-intro} to priors with exponential tails.
 For $\alpha, \sigma>0$, define
 \begin{equation}
 \label{eq:Galpha}
     \calG_{\alpha}(\sigma) = \left\{\Law(U):
     \Expect\!\left[\exp\!\left((|U|/\sigma)^\alpha\right)\right] \leq 2\right\}.
 \end{equation}
 In the special case of $\alpha=2$, this is the class of subgaussian distributions with
 proxy variance  proportional to $\sigma^2$; for $\alpha = 1$, it is the
 class of subexponential distributions.

 \begin{theorem}
 \label{thm:ub-exp}
 Let $G,H\in \calG_\alpha(\sigma)$.
 There exists a constant $C=C(\alpha,\sigma)$ such that
     \[
        \Regret(H\|G)
        \leq
        \begin{cases}
C   \epsilon^2 \pth{\log \frac{1}{\epsilon}}^{2/\alpha}
        \pth{\log\log \frac{1}{\epsilon}}^2            & \alpha\leq 2 \\
         C   \epsilon^2 \log \frac{1}{\epsilon}
        \log\log \frac{1}{\epsilon}
            & \alpha>2
        \end{cases}
      \]
      where $\epsilon^2 = \hellinger^2(f_{G},f_{H})$.
 \end{theorem}

For heavy-tailed priors---for instance, those satisfying only a bounded $p$th-moment
condition---we show in \prettyref{app:moment} that regret can be polynomially larger than
the squared Hellinger distance. Hence, the regularization of \cite{JiangZhang2009} is not
merely a proof device but is in fact necessary in order to obtain a regret nearly linear
in $\hellinger^2$.

\subsection{Statistical applications}
\label{sec:statsapp}
Next we discuss implications of these
regret--Hellinger inequalities for empirical Bayes (EB).
In the EB setting, we observe $y_1,\ldots,y_n \iiddistr f_G$, where $G$ is the unknown prior of the latent means.
One of the most popular methods for learning the prior is the nonparametric maximum
likelihood estimator (NPMLE) \cite{KW56}:
\begin{equation}
\GNPMLE
= \argmax_{H \in \calP(\reals)} \sum_{i=1}^n \log f_H(y_i).
    \label{eq:NPMLE}
\end{equation}
where the maximization is over all probability measures on $\reals$.
It is known that \prettyref{eq:NPMLE} admits a unique and discrete maximizer, which has at most $n$ atoms in the worst case  and $O(\log n)$ atoms in the typical case when $y_i$'s are drawn independently from a subgaussian distribution \cite{simar1976maximum,lindsay1983geometry1,PW20-npmle}.

Suppose the NPMLE computed based on a training set $y_1,\ldots,y_n$ is applied to a fresh test data point. We are interested in bounding the resulting regret (excess over the Bayes squared error).
To simplify presentation and compare with existing results in the literature, we consider expected regret and focus on priors that are either compactly supported or subgaussian.

\paragraph{Subgaussian priors}

The following result improves over the state of the art $O((\log n)^5/n)$ in 
\cite{JiangZhang2009}, derived using regularization.

\begin{corollary}
\label{cor:NPMLE}
Let $y_1,\ldots,y_n \iiddistr f_G$, 
the expectation over which is denoted by $\Expect_G$.
For any $\sigma>0$, there exists $C=C(\sigma)$ such that for any $G \in \calG_2(\sigma)$ that is subgaussian with constant $\sigma$,
\begin{equation}
\Expect_G \Regret(\GNPMLE\|G) \leq \frac{
C (\log n)^3 (\log \log n)^2}{n}.
    \label{eq:reg-NPMLE}
\end{equation}
\end{corollary}
\noindent The proof of \prettyref{cor:NPMLE} is a direct consequence of 
\prettyref{thm:ub-exp}, once $\GNPMLE$ is shown to be $O(1)$-subgaussian with high probability. This is done in \prettyref{app:NPMLE}.

\paragraph{Compactly supported priors}

Here, much stronger results can be obtained. Using \prettyref{thm:ub-compact-intro}, next
we determine the optimal rate of EB regret up to constant factors and show that NPMLE achieves this within a $\log \log n$ factor.

To this end, we first obtain new results on estimating a Gaussian mixture density with a
compactly supported mixing distribution. 
\begin{itemize}
	\item 
We show that the optimal rate for the square Hellinger error is given by $\Theta(\frac{1}{n} \frac{\log n}{\log\log n})$, achieved by, for example, the Le Cam--Birg\'e estimator, a theoretical construction based on pairwise comparison \cite{Birge83}.
This result improves upon \cite{NW21} and resolves conjectures in \cite{jia.polyanskiy.wu.2023} and in a previous
version of this paper.

\item
Furthermore, we show that the NPMLE, when restricted to optimize over distributions with a
given compact support, achieves a square Hellinger risk of $O (\frac{\log n} {n})$, which
is optimal within a $\log \log n$ factor. This result improves the best known result of
$O(\frac{(\log n)^2}{n})$ for NPMLE (see \cite[Theorem 4.1]{ghosal.vdv}, \cite[Theorem
1]{Zhang_2009}, or \cite[Theorem 6]{PW20-npmle}).
\end{itemize}
Due to space limitations, the results on density estimation are postponed to Appendix \ref{sec:optimal-density}. They follow from new bounds on local Hellinger covering and bracketing entropy of Gaussian mixtures and may be of independent interest.

Returning to regret for compactly supported priors, \prettyref {thm:ub-compact-intro}
immediately implies that any optimal proper density estimator attains the optimal regret
of $\Theta(\frac{1}{n} (\frac{\log n}{\log\log n})^2)$, matching the minimax lower
bound\footnote{The lower bound in \cite[Theorem 1] {polyanskiy2021sharp} in fact applies
to any estimator of the true Bayes regressor (or score) that need not be of the proper
form \prettyref{eq:tweedie}. Thus \prettyref {cor:lcb-regret} is a strong endorsement of
the optimality of the $g$-modeling approach to empirical Bayes.} in \cite[Theorem
1]{polyanskiy2021sharp}. In addition, the regret of constrained NPMLE is optimal within a
$\log \log n$ factor. These results are formalized by the following corollaries.

\begin{corollary}[Minimax rate for regret]
\label{cor:lcb-regret}
For any constant $M>0$,
\[
\inf_{\hat G} \sup_{G\in\mathcal P([-M,M])} \mathbb E_G \Regret(\hat G\|G) = \Theta_M
\pth{
\frac{1}{n}
\left(\frac{\log n}{\log\log n}\right)^2
}.
\]
\end{corollary}

\begin{corollary}[Regret of constrained NPMLE]
\label{cor:npmle-regret}
Consider a constrained version of \prettyref{eq:NPMLE}:    %
\begin{equation}
\GNPMLE'
= \argmax_{H \in \calP([-M',M'])} \sum_{i=1}^n \log f_H(y_i).
    \label{eq:NPMLE-constrained}
\end{equation}
where $M$ and $M'$ are constants with $M'\geq M$. There exists $C=C(M')$ such that for any
$G\in\mathcal P([-M,M])$,
\[
\mathbb E_{G}\Regret(\GNPMLE'\|G)
\le
C \frac{(\log n)^2}{n\log\log n}.
\]
\end{corollary}

\section{Compactly supported priors}
\label{sec:compact}

We prove a more precise version of \prettyref{thm:ub-compact-intro}
that is applicable to growing $M$. 

    \begin{theorem}
      \label{thm:ub-compact}
      Let $G,H$ be supported on $[-M,M]$.
      Let $f_G = G * N(0,1)$ and $f_H = H * N(0,1)$ and let $\epsilon^2 = \hellinger^2
      (f_G, f_H)$.
      Assume that the second moments of $G$ and $H$ are at most $\sigma^2$.
      There exists an absolute constant $C_0$ such that
      \[
        \Regret(H\|G)
        \leq
        C_0 e^{\sigma^2/2}  \cdot
        \begin{cases}
          (M+1)^2 \cdot
          \epsilon^2 \frac{\log \frac{1}{\epsilon}}{\log \frac{\log \frac{1}{\epsilon}}
          {M^2}} &
          M^2 \leq \frac{1}{10} \log\frac{1}{\epsilon} \\
          M^4 \epsilon^2 &
          M^2 \geq \frac{1}{10} \log\frac{1}{\epsilon}
        \end{cases}. 
      \]
    \end{theorem}

    \subsection{A first reduction}

    \begin{lemma}
      \label{lmm:reduction}
      \[
      \Regret(H\|G)
        \leq C_0 M^2 \delta + C_0 \Delta
      \]
      for some absolute constant $C_0$, where
 \begin{align}
      \delta &= 2 \int_\reals \frac{(f_G-f_H)^2}{f_G+f_H} \\
      \Delta &= 2 \int_\reals \frac{(m_Gf_G-m_Hf_H)^2}{f_G+f_H}.
    \end{align}
    \end{lemma}
    \begin{proof}
      Write
      \begin{align}
        \frac{f'_G}{f_G} - \frac{f'_H}{f_H}
        = & ~  \frac{f_G'}{f_G}-\frac{2f_G'}{f_G+f_H} + \frac{2(f_G'-f_H')}{f_G+f_H} +
        \frac{2f_H'}{f_G+f_H}  - \frac{f_H'}{f_H}\nonumber\\
        = & ~   \frac{f_G'(f_H-f_G)}{f_G(f_G+f_H)} + \frac{2(f_G'-f_H')}{f_G+f_H} +
        \frac{f_H'(f_H-f_G)}{f_H(f_G+f_H)} \nonumber \\
        = & ~   (m_G+m_H)
        \frac{(f_H-f_G)}{f_G+f_H} + 2 \frac{m_G f_G-m_H f_H}{f_G+f_H}.
        \label{eq:decomp}
      \end{align}
      The proof is completed by noticing that   $|m_G|,|m_H|\leq M$.
    \end{proof}
Note that $\delta$ is on par with the squared Hellinger distance
$\epsilon^2$:
\begin{equation}
2 \epsilon^2 \leq \delta \leq 4\epsilon^2.
\label{eq:delta-eps}
\end{equation}
Thus, \prettyref{lmm:reduction} reduces the problem to proving
    \[
      \Delta \leq C \delta \frac{\log \frac{1}{\delta}}{\log\log \frac{1}{\delta}},
    \]
    while retaining dependence of the constant $C$ on $M$ and $\sigma$.

    \subsection{Polynomial approximation and a Bernstein-style inequality}

    Let $f \equiv (f_G+f_H)/2= f_{(G+H)/2}$.
Define the density ratio
    \begin{equation}
      g = \frac{f_G-f_H}{\varphi}
      \label{eq:g}
    \end{equation}
    where $\varphi$ is the standard normal density.
    Define the weight function
\begin{equation}
w = \frac{\varphi^2}{f}.
    \label{eq:w}
\end{equation}
    Crucially, we have
    \begin{align}
      \delta &= \int_{\reals} g^2 w = \|g\|_{L_2(w)}^2 \label{eq:delta}\\
      \Delta &= \int_{\reals} (g'(x))^2 w(x)\,dx = \|g'\|_{L_2
      (w)}^2
      \label{eq:Delta}.
    \end{align}
    To estimate the differential operator on $g$, we first approximate $g$ by a polynomial, bound the operator norm of
    the differential operator on polynomials by a linear algebraic argument, then quantify the approximation error bound. These are carried out in the following two lemmas.

    \begin{lemma}[Bernstein-style inequality for weight $w$]
      \label{lmm:bernstein}
      Let $p$ be a degree-$k$ polynomial.  Then
      \[
        \|p'\|_{L_2(w)}
        \leq (2M+1)\sqrt{k+1} \|p\|_{L_2(w)}.
      \]
    \end{lemma}

    \begin{lemma}[Polynomial approximation]
      \label{lmm:poly-f}
      Let $g$ be given in \prettyref{eq:g}. Assume that $\nu = (G+H)/2$ has zero mean.
      For any $k \geq 2e M^2$,
      there exists a degree-$k$ polynomial $p$ such that
      \[
        \|g-p\|_{L_2(w)}^2 \maxwith
        \|g'-p'\|_{L_2(w)}^2
        \lesssim e^{\sigma^2/2} M^2 \pth{\frac{eM^2}{k}}^k.
      \]
    \end{lemma}

\subsection{Complete the proof}
With Lemmas \ref{lmm:reduction}--\ref{lmm:poly-f}, we now prove \prettyref{thm:ub-compact}.

\begin{proof}
      Without loss of generality, we may assume $\delta \leq 0.01$.
      In addition, we may assume $\nu=\frac{G+H}{2}$ has zero mean at a price of inflating
      $M$ to $2M$. Indeed, suppose the mean $\mu$ is not zero. Let $\tilde G$ and $\tilde
      H$ be the shifted version of $G$ and $H$ by $-\mu$ respectively, so that $\tilde
      \nu=(\tilde G+\tilde H)/2$ has zero mean.
      Since $\mu \in [-M,M]$,
      $\tilde G,\tilde H$ are both supported on $[-2M,2M]$. Furthermore,
      both the Hellinger distance and the regret are  shift-invariant:
      $\Regret(\tilde H\|\tilde G)= \Regret(H\|G)$ and
      $\hellinger^2(f_{\tilde G},f_{\tilde H}) = \hellinger^2(f_G,f_H)$.

      First, for a degree-$k$ polynomial,  we have
      \begin{align*}
        \|g'\|_{L_2(w)}
        &\leq
        \|g'-p'\|_{L_2(w)} +
        \|p'\|_{L_2(w)}\\
        &\leq
        \|g'-p'\|_{L_2(w)}+
        (2M+1)\sqrt{k+1} \|p\|_{L_2(w)}\\
        &\leq \|g'-p'\|_{L_2(w)}+
        (2M+1)\sqrt{k+1} (\|g\|_{L_2(w)}
        + \|g-p\|_{L_2(w)})
      \end{align*}
      where the second inequality follows from  \prettyref{lmm:bernstein}.
      Combining \prettyref{lmm:bernstein}, \prettyref{lmm:poly-f}, and  
      \prettyref{eq:delta}--\prettyref{eq:Delta}, we obtain
      \[
        \Delta
        \lesssim e^{\sigma^2/2}\pr{
        (M+1)^2 k    \delta+
        M^2 (M+1)^2 k \pth{\frac{eM^2}{k}}^{k}}
      \]
      provided $k\geq 2eM^2$.
      Finally, we optimize $k$.
      If $M^2 \leq 0.1 \log \frac{1}{\delta}$, we choose $k = C_1 \frac{\log \frac{1}{\delta}}{\log \frac{\log \frac{1}{\delta}}{M^2}}$
      for suitably large absolute constant $C_1$.
      If $M^2 \geq 0.1 \log \frac{1}{\delta}$, we choose $k = C_1 M^2$. In both cases, the
      second term in the preceding display is at most $\delta^{10}$. Finally, invoking \prettyref{lmm:reduction} completes the proof.
    \end{proof}

\subsection{Proof of lemmas}

    \paragraph{Proof of \prettyref{lmm:bernstein}}

    Define the inner product weighted by $w$ as:
    \[
      \Iprod{f}{g} \equiv \int_\reals f g w, \quad f,g\in L_2(w).
    \]
    Note that $w$ has at least a Gaussian tail so $L_2(w)$ includes all polynomials.
    Denote by $\{q_j: j\geq 0\}$ the system of orthonormal polynomials for $w$ such that
    \[
      \Iprod{q_i}{q_j} = \indc{i=j}.
    \]
    In particular, $\Iprod{q_i}{p}=0$ for every polynomial of degree at most $i-1$.
    To be definitive, assume that the leading coefficient of each $q_i$ is (strictly) positive, denoted by $\kappa_j>0$.

    Expand a degree-$k$ polynomial $p$ as
    $p=\sum_{j=0}^k a_j q_j$ and
    $p'=
    \sum_{j=0}^k b_j q_j$, where the coefficients are related by vector-matrix multiplication:
    \[
      b = L a
    \quad \text{ and }\quad
      \|p'\|_{L_2(w)}
      \leq \|L\|_{\op} \|p\|_{L_2(w)}.
    \]
    where $ \|L\|_{\op}$ is the largest singular value of the $(k+1)\times (k+1)$ matrix $L$.

    We thus analyze $L$ and bound $\norm{L}_{\op}$. Note that $p'=\sum_{j=0}^k a_j q_j'$, and
    so $L= (L_ {ij})$ is given by expanding each $q_j'$, namely
    \[
      q_j' = \sum_{i=0}^k L_{ij} q_i.
    \]
    By orthogonality,
    $L_{ij} = \Iprod{q_i}{q_j'}$. Clearly, $L_{ij} = 0$ for all $i\geq j$; in other words, $L$ is upper triangular.
    Note that, by integration by parts, we have
    \begin{equation}
      \Iprod{f'}{g} =
      - \Iprod{f}{g'} + \Iprod{f}{gV'} \quad \text{where } \quad w = e^{-V}, \quad V'=-\frac{w'}{w}, 
      \label{eq:adjoint}
    \end{equation}
    for all $f,g$ with at most exponential growth.
    Therefore, \[
      L_{ij} =
      \begin{cases}
        0, & i \ge j\\
        \ip{q_i, q_j V'} &  i < j
      \end{cases},
    \]
    since $\ip{q_i', q_j} = 0$ when $ i < j$.

    To bound the operator norm of $L$, we decompose it into a \textit{dense} part and a \textit{sparse} part.
    Since $w = \varphi^2/f$, $f= f_\nu$ with $\nu=\frac{G+H}{2}$ supported on $[-M,M]$, we have
    \begin{equation}
      V'(y) = -(\log w)'(y) = \frac{f_\nu'}{f_\nu}(y) + 2 y = m_\nu(y) + y
      \label{eq:Vprime}
    \end{equation}
    Then
    $L=A+B$,
    where
    \[
      A_{ij} =
      \begin{cases}
        \Iprod{q_i}{q_j m_\nu}    & i < j\\
        0  &  \text{else}
      \end{cases}
      \quad
      B_{ij} =
      \begin{cases}
        \Iprod{q_{j-1}}{q_j y}    & i = j-1\\
        0   & \text{else}
      \end{cases}
    \]
    where we used the fact that,  by orthogonality,
    $\Iprod{q_i}{q_j y}=
    \Iprod{q_i y}{q_j}
    =0$ if $i<j-1$.
    It remains to bound the operator norms of $A$ and $B$ separately.

    For $A$, we compute its Frobenius norm:
    \[
      \|A\|_F^2=
      \sum_{j=1}^k \sum_{i=0}^{j-1}
      \Iprod{q_i}{q_j m_\nu}^2
      \leq \sum_{j=1}^k \|q_j m_\nu\|_{L_2(w)}^2
      \leq M^2 \sum_{j=1}^k \| {q_j}\|_{L_2(w)}^2 = M^2 k.
    \]
    Thus $\|A\|_{\op} \leq M \sqrt{k}$.

    Note that $B$ contains a single superdiagonal. Then
    \[
      \|B\|_{\op} = \max_{j=1,\ldots,k} |a_j|
    \]
    where
    \[
      a_j \equiv B_{j-1,j} =
      \Iprod{q_{j-1}}{q_j y} = \ip{q_{j-1}y, q_j}.
    \]
    Equivalently, $a_j$ appears in the three-term recursion:
    \begin{align}
      y q_j(y) = a_{j+1} q_{j+1}(y)
      + b_j q_j(y) + a_j q_{j-1}(y)
    \end{align}
    for some $b_j$.
    By assumption
    $q_j(y)=\kappa_j y^j +\ldots$ and $q_{j-1}(y)=\kappa_{j-1} y^{j-1} +\ldots$ with
    $\kappa_j>0$. Thus $q_j = \frac{\kappa_j}{\kappa_{j-1}}y q_{j-1} + p_{j-1}$ for some
    polynomial of degree at most $j-1$.
    Comparing leading coefficients, we thus get
    $a_j =\Iprod{q_j }{q_{j-1} y}
    = \frac{\kappa_{j-1}}{\kappa_j}>0$.
    By a similar argument, $\Iprod{q_j'}{q_{j-1}} = \frac{j\kappa_j}{\kappa_{j-1}}$.
    Applying \prettyref{eq:adjoint} and \prettyref{eq:Vprime},
    \[
      \underbrace{\Iprod{q_j'}{q_{j-1}}}_{j/a_j} =
      -\underbrace{\Iprod{q_j}{q_{j-1}'}}_{=0} +
      \Iprod{q_j }{q_{j-1} V'}
      = \underbrace{\Iprod{q_j }{q_{j-1} m_\nu}}_{\equiv \beta_j} + \underbrace{\Iprod{q_j }{q_{j-1}y}}_{a_j}
    \]
    Thus we get
    $a_j(a_j+\beta_j) = j$, that is,
    $a_j = (-\beta_j + \sqrt{\beta_j^2 + 4 j})/2$. Applying Cauchy-Schwarz,
    \[
      |\beta_j| = |\Iprod{q_j }{q_{j-1} m_\nu}| \leq
      \|q_j\|_{L_2(w)} \|q_{j-1}\|_{L_2(w)}
      \|m_\nu\|_\infty \leq M.
    \]
    and hence $a_j \leq \sqrt{j} + M$.
    This shows
    \[
      \|B\|_{\op} \leq \sqrt{k}+M
    \]
    and hence $\|L\|_{\op} \leq \sqrt{k}+M + M\sqrt{k}$, as desired.

    \begin{remark}
    \label{rmk:bernstein-perturb}
      The same argument shows that the weight $w$ satisfies the following Bernstein-style inequality
      \[
        \|p'\|_{L_2(w)}
        \leq O(\sqrt{k}) \|p\|_{L_2(w)}
      \]
      for all
      degree-$k$ polynomial $p$, provided that the score function
      $V'= -(\log w)'$ is a bounded perturbation to a linear function
      \[
        |V'(y) - c y| \leq C
      \]
      for some constant $c,C > 0$.
      This can be viewed as a perturbation result to the Bernstein inequality for Gaussians.\footnote{This is an immediate consequence of the derivative of Hermite polynomials: $H_j'=j H_{j-1}$ \cite{orthogonal.poly}.} For example, for $w=\varphi$,
      \[
        \|p'\|_{L_2(\varphi)}
        \leq \sqrt{k} \|p\|_{L_2(\varphi)},
      \]
      attained at Hermite polynomials.
    \end{remark}

    \paragraph{Proof of \prettyref{lmm:poly-f}}
    Since $\nu=(G+H)/2$ is centered, we have by Jensen's inequality
    \[
      f(y) = \Expect_{U\sim \nu}[\varphi(y-U)]
      \geq \exp(-\sigma^2/2) \varphi(y)
    \]
    where $\Expect_{U\sim \nu}[U^2] \leq \sigma^2$ by the assumption that the second moments of $G$ and $H$ are both at most $\sigma^2$.
    Thus
    \begin{equation}
      w(y) \leq \exp(\sigma^2/2) \varphi(y)
        \label{eq:w-phi}
    \end{equation}
    So it suffices to bound the approximation error of $g$ and $g'$ in $L_2(\varphi)$, under standard Gaussian.

    It is well-known that the likelihood ratio between a Gaussian mixture and a standard
    Gaussian admits the following Hermite expansion (see, e.g., \cite[Lemma 9]{WY18}):
    \[
      \frac{f_G(y)}{\varphi(y)} = \sum_{j\geq 0} \frac{m_j(G)}{j!} H_j(y)
    \]
    where $H_j$ is the degree-$j$ Hermite polynomial satisfying $\Expect[H_i(Z)H_j(Z)]=i! \indc{i=j}$, and $m_j(G)$ is the $j$th moment of $G$.
    Hence
    \[
      g(y) =
      \frac{f_G(y)-f_H(y)}{\varphi(y)} = \sum_{j\geq 0} \frac{m_j(G)-m_j(H)}{j!} H_j(y).
    \]
    Truncate this series at $k$ and denote it by $p(y)$. Then we also have
    $g'(y)-p'(y) =
    \sum_{j>k} \frac{m_j(G)-m_j(H)}{j!} H_j'(y) =
    \sum_{j>k} \frac{m_j(G)-m_j(H)}{(j-1)!} H_{j-1}(y)$.
    Thus
    \[
      \|g-p\|_{L_2(\varphi)}^2 = \sum_{j>k} \frac{|m_j(G)-m_j(H)|^2}{j!},
      \quad
      \|g'-p'\|_{L_2(\varphi)}^2 = \sum_{j>k} \frac{|m_j(G)-m_j(H)|^2}{(j-1)!}
    \]
    and it suffices to bound the latter.
    Note that $|m_j(G)|,|m_j(H)| \leq M^j$ and $j! \geq (j/e)^j$, we get
    \begin{equation}
      \|g'-p'\|_{L_2(\varphi)}^2
      \lesssim M^2 \sum_{\ell\geq k} \pth{\frac{eM^2}{k}}^\ell
      \label{eq:gp-compact}
    \end{equation}
    completing the proof by summing the geometric series.

\section{Priors with exponential tails}
\label{sec:exp}

This section proves \prettyref{thm:ub-exp}, extending the regret bound in 
\prettyref{thm:ub-compact} for compactly supported priors to those in the class
$\calG_\alpha(\sigma)$ defined in \prettyref{eq:Galpha}, having exponential tails. Throughout this section, all constants (including those implicit in $\lesssim$) depend on $\alpha,\sigma$.

Note that one may apply \prettyref{thm:ub-compact} by approximating the true
prior with a compactly supported one.  This yields a worse logarithmic factor. For
example, for the subgaussian class, applying
\prettyref{thm:ub-compact} with $M \asymp \sqrt{\log \frac{1}{\epsilon}}$ yields an
$O\pr{\epsilon (\log \frac{1}{\epsilon})^2}$ upper bound, which is inferior to
\prettyref{thm:ub-exp}. Instead, we take a more direct approach.

    Similar to the proof of \prettyref{thm:ub-compact}, the first step is again to reduce  $\Regret(H\|G)$ to a weighted derivative quantity $\Delta$.
    Recall the definitions of $g$, $w$, $\delta$, and $\Delta$ in
    \eqref{eq:g}--\eqref{eq:Delta} from \prettyref{sec:compact}.

    \begin{prop}[Reduction to $\Delta$]
      \label{prop:exp-delta-reduction}
      For $H,G \in \mathcal{G}_\alpha(\sigma)$,
      \[
        \Regret(H\|G)
        \lesssim
        \Delta
        +
        \delta \pth{\log \frac{1}{\delta}}^{2/(\alpha\wedge 2)}.
      \]
    \end{prop}

    \begin{proof}
    By the
      same decomposition \prettyref{eq:decomp} in  \prettyref{lmm:reduction},
      \[
        \Regret(H\|G)
        \lesssim
        \int_\reals (m_G+m_H)^2 \frac{(f_G-f_H)^2}{f_G+f_H}
        +
        \int_\reals \frac{(m_Gf_G-m_Hf_H)^2}{f_G+f_H}.
      \]
      By \eqref{eq:Delta}, the second term is $\Delta/2$, so
      \[
        \Regret(H\|G)
        \lesssim
        \int_\reals (m_G+m_H)^2 \frac{(f_G-f_H)^2}{f_G+f_H}
        + \Delta.
      \]

      To bound the posterior mean, we apply \eqref{eq:V1} from \prettyref{lem:Vpp-phi2-over-f},
      with $a:=\sigma(\log 4)^{1/\alpha}$ satisfying $F([-a,a])\geq 1/2$ for  $F\in\calG_\alpha(\sigma)$ using Markov's inequality,
      obtaining
      \[
        |m_F(y)| \leq C(1+|y|),
        \qquad
        y\in\reals. \numberthis \label{eq:posterior-mean-bound}
      \]
Fix $T>1$ to be specified next. Then $
        |m_G(y)|+|m_H(y)|
        \lesssim
        T$ for any $|y|\leq T$,
      so
      \[
        \int_{|y|\leq T} (m_G+m_H)^2 \frac{(f_G-f_H)^2}{f_G+f_H}
        \lesssim
        T^2 \delta.
      \]
      For $|y| > T$, using
      $
        \frac{(f_G-f_H)^2}{f_G+f_H} \leq f_G+f_H$
      and \eqref{eq:posterior-mean-bound} again we obtain
      \[
        (m_G(y)+m_H(y))^2 (f_G(y)+f_H(y))
        \lesssim
        y^2(f_G(y)+f_H(y)).
      \]
      Thus
      \[
        \Regret(H\|G)
        \lesssim
        T^2 \delta + \Delta + \int_{|y|>T} y^2(f_G+f_H).
      \]
Since $G\in \calG_\alpha(\sigma)$ and $f_G = G * N(0,1)$, by a Chernoff bound, the tail contribution is at most
\[
\int_{|y|>T} y^2 f_G \lesssim T^2 \exp(-c T^{\alpha \wedge 2})
\]
for some constant $c$, and the same holds for $f_H$. Finally, choosing \[
        T:=C \pth{\log\pth{1+\frac{1}{\delta}}}^{1/(\alpha\wedge 2)},
      \]
      for appropriately large $C$,
      we have
     $ \int_{|y|>T} y^2(f_G+f_H) \leq \delta^2$, completing the proof.
\end{proof}

    The next lemma is an analogue of the polynomial approximation in \prettyref{lmm:poly-f}.

    \begin{lemma}
      \label{lmm:exp-poly}
      Let $g = \frac{f_G-f_H}{\varphi}$
      be given in \prettyref{eq:g}.
      Assume that $\nu=(G+H)/2$ has mean zero for $G, H \in \mathcal G_\alpha(\sigma)$. Then there exist constants
      $c,C, k_0>0$, depending only on $\alpha,\sigma$, such that for every $k\geq k_0$
      there is a degree-$k$ polynomial $p_k$ with
      \[
        \|g-p_k\|_{L_2(w)}^2
        \vee
        \|g'-p_k'\|_{L_2(w)}^2
        \leq
        \begin{cases}
          C \exp\pth{-c k^{\alpha/2}},
          & \alpha\leq 2,\\
          C \exp\pth{-c k \log k},
          & \alpha>2.
        \end{cases}
      \]
    \end{lemma}

    Finally, the following result is an analogue of \prettyref{lmm:bernstein}.
    The proof is deferred to the end of \prettyref{app:bernstein}.

    \begin{prop}
      \label{prop:exp-bernstein}
      Assume that $\nu=(G+H)/2$ has mean zero for $G, H \in \mathcal G_\alpha(\sigma)$. Then every degree-$k$ polynomial
      $p$ with $k\geq 2$ satisfies
      \[
        \|p'\|_{L_2(w)}
        \lesssim
        \sqrt{k}\,\log k\,\|p\|_{L_2(w)}.
      \]
    \end{prop}

    Putting these together, we prove \prettyref{thm:ub-exp}.

  \begin{proof}[Proof of \prettyref{thm:ub-exp}]
    Since
    \[
      \epsilon^2
      =
      \int_\reals \frac{(f_G-f_H)^2}{(\sqrt{f_G}+\sqrt{f_H})^2}
      \asymp
      2\int_\reals \frac{(f_G-f_H)^2}{f_G+f_H}
      =
      \delta = \|g\|_{L_2
      (w)}^2,
    \]
    it suffices to work with $\delta$. As in the proof of \prettyref{thm:ub-compact}, it is without loss of generality to assume
    $
      \delta\leq e^{-e},
    $
    and
    $\nu=(G+H)/2$ is mean-zero by inflating the constant $\sigma$.

    By \prettyref{prop:exp-delta-reduction}, it remains to bound $\Delta=\|g'\|_{L_2
      (w)}^2$.
    Let $p_k$ be the polynomial from \prettyref{lmm:exp-poly}. By
    \prettyref{prop:exp-bernstein},
    \[
      \|g'\|_{L_2(w)}
      \leq
      \|g'-p_k'\|_{L_2(w)}
      +
      C_{\alpha,\sigma}\sqrt{k+1}\log k
      \pth{
        \|g\|_{L_2(w)}+\|g-p_k\|_{L_2(w)}
      }.
    \]
    That is,
    \[
      \Delta
      \lesssim
      k\log^2 k\,\bigl(\delta+E_k\bigr),
    \]
    where $E_k
      :=
      \|g-p_k\|_{L_2(w)}^2
      \vee
      \|g'-p_k'\|_{L_2(w)}^2$.
      Applying  \prettyref{lmm:exp-poly}, setting
    \[
      k=
      \begin{cases}
      \Big\lceil
        C_{\alpha,\sigma}
        \log^{2/\alpha}\pth{\frac{e}{\delta}}
      \Big\rceil & \alpha\leq 2\\
      \Big\lceil
        C_{\alpha,\sigma}
        \frac{\log\pth{\frac{e}{\delta}}}
        {\log\pth{e+\log\pth{\frac{e}{\delta}}}}
      \Big\rceil & \alpha > 2\\
      \end{cases}
    \]
    with suitably large $C_{\alpha,\sigma}$ ensures
    $E_k\leq \delta$ and completes the proof.
  \end{proof}

\begin{proof}[Proof of \prettyref{lmm:exp-poly}]
Let
  \[
    L:=
    \begin{cases}
      c \sqrt{k},
      & \alpha\leq 2,\\
      c \bigl(k\log k\bigr)^{1/\alpha},
      & \alpha>2,
    \end{cases}
  \]
  where $c>0$ is a sufficiently small constant depending on $\alpha$.
  Let $k$ be sufficiently large so that $L\geq 1$.

  Write
  \[
    G=G_{\leq L}+G_{>L},
    \qquad
    H=H_{\leq L}+H_{>L},
  \]
  where $G_{\leq L}$ and $H_{\leq L}$ restrict $G$ and $H$ to $[-L,L]$, and set
  \[
    g_{\leq L}
    :=
    \frac{f_{G_{\leq L}}-f_{H_{\leq L}}}{\varphi},
    \qquad
    g_{>L}
    :=
    \frac{f_{G_{>L}}-f_{H_{>L}}}{\varphi}.
  \]
  Then $
    g=g_{\leq L}+g_{>L}$.

  For the compactly supported part $g_{\le L}$, the argument is identical to that of 
  \prettyref{lmm:poly-f}. Since $\nu$ is centered and belongs to $\calG_\alpha(\sigma)$, its second
  moment is at most a constant depending only on $\alpha,\sigma$.
  Applying \prettyref{eq:w-phi}, we have $w \lesssim \varphi$ everywhere.
Thus \prettyref{eq:gp-compact} yields a degree-$k$ polynomial $p_k$ such that
  \[
    \|g_{\leq L}-p_k\|_{L_2(w)}^2
    \vee
    \|g_{\leq L}'-p_k'\|_{L_2(w)}^2
    \lesssim
    L^2 \pth{\frac{eL^2}{k}}^k
  \]
  provided that $k \geq 2e L^2$.
  By the choice of $L$,
  \[
    L^2 \pth{\frac{eL^2}{k}}^k
    \leq
    \begin{cases}
      C\exp(-c k),
      & \alpha\leq 2,\\
      C\exp\pth{-c k\log k},
      & \alpha>2.
    \end{cases}
  \]

  It remains to control the tail part $g_{\ge L}$. Unlike the proof of 
  \prettyref{lmm:poly-f}, we cannot use $w \lesssim \varphi$ and bound $L_2(\varphi)$, because $\|g\|_{L_2(\varphi)}$ may be infinite (e.g.~for Gaussian $G$ with sufficiently large variance.)
  Instead, we directly control the tail in $L_2(w)$, where $w=\frac{\varphi^2}{f}$. Since
  $f=\frac{f_G+f_H}{2}$, we have
  \begin{align*}
    \|g_{>L}\|_{L_2(w)}^2
    &\leq
    2 \int \frac{f_{G_{>L}}(y)^2}{f(y)}\,dy
    +
    2 \int \frac{f_{H_{>L}}(y)^2}{f(y)}\,dy\\
    &\leq
    4 \int \frac{f_{G_{>L}}(y)^2}{f_G(y)}\,dy
    +
    4 \int \frac{f_{H_{>L}}(y)^2}{f_H(y)}\,dy\\
    &\leq
    4 \int f_{G_{>L}}(y)\,dy
    +
    4 \int f_{H_{>L}}(y)\,dy\\
    &=
    4 G(|U|>L)+4 H(|U|>L) \lesssim
    \exp(-c L^\alpha)
  \end{align*}
  the last step applying Chernoff's bound.

  For the derivative, note that
  \[
   \left(\frac{f_{G_{>L}}}{\varphi}\right)(y)
  = \int \exp\pth{yu-\frac{u^2}{2}}\,G_{>L}(du), \qquad
   \left(\frac{f_{G_{>L}}}{\varphi}\right)'(y)
  = \int u \exp\pth{yu-\frac{u^2}{2}}\,G_{>L}(du).
  \]
  Thus
  \[
  \left(\frac{f_{G_{>L}}}{\varphi}\right)'(y)
  =
  \frac{f_{G_{>L}}(y)}{\varphi(y)}
  \int u \pi_y(du), \quad
  \pi_y(du) := \frac{ \exp\pth{yu-\frac{u^2}{2}}
G_{>L}(du)
  }{\int \exp\pth{yu-\frac{u^2}{2}}\,G_{>L}(du)}
  \]
  where $\pi_y$ is a probability measure. Thus Cauchy-Schwarz gives
  \[
    \left(\frac{f_{G_{>L}}}{\varphi}\right)'(y)^2
    \leq
    \frac{f_{G_{>L}}(y)}{\varphi(y)}
    \int u^2 \exp\pth{yu-\frac{u^2}{2}}\,G_{>L}(du),
  \]
  and similarly for $H_{>L}$. Therefore
  \begin{align*}
    \|g_{>L}'\|_{L_2(w)}^2
    &\leq
    2 \int \left(\frac{f_{G_{>L}}}{\varphi}\right)'(y)^2
    \frac{\varphi(y)^2}{f(y)}\,dy
    +
    2 \int \left(\frac{f_{H_{>L}}}{\varphi}\right)'(y)^2
    \frac{\varphi(y)^2}{f(y)}\,dy\\
    &\leq
    4 \int \left(\frac{f_{G_{>L}}}{\varphi}\right)'(y)^2
    \frac{\varphi(y)^2}{f_G(y)}\,dy
    +
    4 \int \left(\frac{f_{H_{>L}}}{\varphi}\right)'(y)^2
    \frac{\varphi(y)^2}{f_H(y)}\,dy\\
    &\leq
    4 \int_{|u|>L} u^2\,G(du)
    +
    4 \int_{|u|>L} u^2\,H(du) \lesssim
    L^2 \exp(-c L^\alpha),
  \end{align*}
  the last step again applying Chernoff's bound.

  By the choice of $L$, the tail part satisfies
  \[
    \|g_{>L}\|_{L_2(w)}^2
    \vee
    \|g_{>L}'\|_{L_2(w)}^2
    \leq
    \begin{cases}
      C' \exp(-c k^{\alpha/2}),
      & \alpha\leq 2,\\
      C' \exp\pth{-c k\log k},
      & \alpha>2.
    \end{cases}
  \]
  Finally, the proof is completed with
  \[
    \|g-p_k\|_{L_2(w)}^2
    \leq
    2\|g_{\leq L}-p_k\|_{L_2(w)}^2
    +
    2\|g_{>L}\|_{L_2(w)}^2,
  \]
  and similarly for the derivative.
\end{proof}

    \appendix

    \section{Proof of \prettyref{thm:lb}}
\label{app:lb}
\begin{proof}
  Let $\nu$ denote the arcsine law on $[-1,1]$:
  \[
d\nu(x)=\frac{\indc{x \in [-1,1]}}{\pi \sqrt{1-x^2}}\,dx.
  \]
  Let $\nu_m$ be the $m$-point Gauss quadrature of $\nu$, supported on the Chebyshev nodes
  \[
    x_{m,j}:=\cos\frac{(2j-1)\pi}{2m},\qquad 1\le j\le m.
  \]
  These are roots of the degree-$m$ Chebyshev polynomial
      $T_m(x) := \cos(m\arccos(x))$
      which are orthogonal under $\nu$.
 Then $\nu_m$ and $\nu$ match moments through order $2m-1$, i.e.,
  \begin{equation}
    \int p(x)\,\nu(dx)=\int p(x)\,\nu_m(dx)
    \qquad
    \text{for every polynomial $p$ with $\deg p\le 2m-1$.}
    \label{eq:quad-exact}
  \end{equation}
  Define
  \[
    G_m:=(1-\tau_m)\delta_0+\tau_m \nu,
    \qquad
    H_m:=(1-\tau_m)\delta_0+\tau_m \nu_m,
  \]
  where $\tau_m\in(0,1/2)$ is a vanishing sequence to be specified later.

In order to bound the Hellinger distance and the regret, define
  \[
    U_m(y):=\frac{f_\nu(y)-f_{\nu_m}(y)}{2\varphi(y)}
    =
    \frac12\int e^{uy-u^2/2}\,(\nu-\nu_m)(du), \qquad
    V_m(y):=\frac{f_\nu(y)+f_{\nu_m}(y)}{2\varphi(y)}-1.
  \]
  and
  \[
  \alpha_m :=  \int U_m(y)^2\varphi(y)\,dy,
  \quad
   \beta_m:=\int U_m'(y)^2\varphi(y)\,dy
  \]
Then
  \[    f_{G_m}=\varphi\bigl(1+\tau_m(V_m+U_m)\bigr),
    \qquad
    f_{H_m}=\varphi\bigl(1+\tau_m(V_m-U_m)\bigr).
  \]
   Since $\nu$ and $\nu_m$ are probability measures on $[-1,1]$, for every $y\in\reals$,
  \begin{equation}
    |U_m'(y)|\le e^{|y|},
    \qquad
    |V_m(y)|\le e^{|y|}+1.
    \label{eq:UV-bounds}
  \end{equation}

  The Hellinger distance may be bounded from above as follows: By construction, we have $f_{G_m},f_{H_m}\ge \varphi/2$ and
  \begin{equation}
    \hellinger^2(f_{G_m},f_{H_m})
    =
    \int \frac{(f_{G_m}-f_{H_m})^2}{(\sqrt{f_{G_m}}+\sqrt{f_{H_m}})^2}
    \le
    \int \frac{(f_{G_m}-f_{H_m})^2}{\varphi}
    = \tau_m^2 \int \frac{(f_{\nu}-f_{\nu_m})^2}{\varphi}
    =
    4\tau_m^2 \alpha_m.
      \label{eq:H2-upper}
  \end{equation}

For the regret, applying the decomposition \prettyref{eq:decomp} in the proof of
\prettyref{lmm:reduction} and \prettyref{eq:Delta}, we get the lower bound
\[
\Regret(H_m\|G_m)
\geq c \tau_m^2 \|U_m'\|_{L_2(w)}^2
-C \hellinger^2(f_{G_m},f_{H_m}),
\]
for some absolute constants $c,C$, where $w$ is given by \prettyref{eq:w}, namely
\[
w = \frac{\varphi^2}{(f_{G_m}+f_{H_m})/2}
= \frac{\varphi}{1+\tau_m V_m}.
\]
Then
\[
 \|U_m'\|_{L_2(w)}^2
 \geq  \underbrace{\|U_m'\|_{L_2(\varphi)}^2}_{\beta_m}
 -  \tau_m \int (U_m')^2 \frac{\varphi V_m}{1+\tau_m V_m}.
\]
Applying $1+\tau_m V_m\geq 1-\tau_m \geq 1/2$ and \prettyref{eq:UV-bounds},
we get
\[
\int (U_m')^2 \frac{\varphi V_m}{1+\tau_m V_m}
\leq 4 \int e^{3|y|} \varphi(y)
\]
which is an absolute constant.
In all, we get
\begin{equation}
\Regret(H_m\|G_m) \geq c\tau_m^2 \beta_m - C \tau_m^3 - C \hellinger^2(f_{G_m},f_{H_m}).
    \label{eq:reg-lower}
\end{equation}

Next, we relate $\alpha_m,\beta_m$ to the moment difference. For $j\ge 0$, write
  \[
    \Delta_{m,j}:=\int x^j\,(\nu-\nu_m)(dx).
  \]
  By \eqref{eq:quad-exact}, $
    \Delta_{m,j}=0,\qquad 0\le j\le 2m-1$.
  Moreover, using the leading coefficient of $T_m$, we have
  \[
    x^{2m}=2^{1-2m}T_{2m}(x)+q_{2m-1}(x)
  \]
  for some polynomial $q_{2m-1}$ of degree at most $2m-1$. Then
  \[
    \Delta_{m,2m}
    =
    2^{1-2m}
    \left(
      \int T_{2m}(x)\,\nu(dx)-\int T_{2m}(x)\,\nu_m(dx)
    \right).
  \]
  Now $\int T_{2m}\,d\nu=0$, while $T_{2m}(x_{m,j})=\cos((2j-1)\pi)=-1$, hence $
    \int T_{2m}(x)\,\nu_m(dx)=-1$,
  and therefore
  \begin{equation}
    \Delta_{m,2m}=2^{1-2m}.
    \label{eq:first-gap}
  \end{equation}
  Also, since $\nu$ and $\nu_m$ are probability measures supported on $[-1,1]$,
  \begin{equation}
    |\Delta_{m,j}|\le 2,\qquad j\ge 1.
    \label{eq:moment-bound}
  \end{equation}

  Expanding in Hermite polynomials and using $H_j'=j H_{j-1}$, we have
  \[
    U_m(y)
    =
    \frac12\sum_{j\ge 2m}\Delta_{m,j}\frac{H_j(y)}{j!},
    \quad
    U_m'(y)
    =
    \frac12\sum_{j\ge 2m}\Delta_{m,j}\frac{H_{j-1}(y)}{(j-1)!}.
  \]
  Hence, by Hermite orthogonality,
  \begin{align*}
          \alpha_m&=\int U_m(y)^2\varphi(y)\,dy
    =
    \frac14\sum_{j\ge 2m}\frac{\Delta_{m,j}^2}{j!}\\
    \beta_m&=\int U_m'(y)^2\varphi(y)\,dy
    =
    \frac14\sum_{j\ge 2m}\frac{\Delta_{m,j}^2}{(j-1)!}.  \end{align*}
  In particular,
  \begin{equation}
    \beta_m\ge 2m\,\alpha_m.
    \label{eq:beta-lb}
  \end{equation}
  By \eqref{eq:first-gap} and \eqref{eq:moment-bound},
  \begin{equation}
    \frac{2^{-4m}}{(2m)!}
    \le
    \alpha_m
    \le
    \sum_{j\ge 2m}\frac{1}{j!}
    \le
    \frac{2}{(2m)!}
    \qquad \text{for all large }m.
    \label{eq:alpha-bounds}
  \end{equation}
  Thus $\alpha_m\to 0$, and
  \[
  m \geq c_0 \frac{\log \frac{1}{\alpha_m}}{\log\log \frac{1}{\alpha_m}}
    \]
    for some constant $c_0$.

  Now choose $\tau_m:=\alpha_m^2$ (so as to make the middle term in \prettyref{eq:reg-lower} negligible.)
  We have from \prettyref{eq:H2-upper}
  \[
   \epsilon_m^2 := \hellinger^2(f_{G_m},f_{H_m})
    \leq
    4 \alpha_m^5.
  \]
  Therefore
  \begin{equation}
    \hellinger^2(f_{G_m},f_{H_m})\le 4\alpha_m^5.
    \label{eq:H-upper}
  \end{equation}
and from \prettyref{eq:reg-lower} and \prettyref{eq:beta-lb} we have
\[
\Regret(H_m\|G_m)
\geq c_1
\alpha_m^5
 \frac{\log \frac{1}{\alpha_m}}{\log\log \frac{1}{\alpha_m}}
 \geq c_2 \epsilon_m^2
 \frac{\log \frac{1}{\epsilon_m}}{\log\log \frac{1}{\epsilon_m}}.
\]
  This proves the theorem.
\end{proof}

    \section{Bernstein inequality for general weights}

\label{app:bernstein}
Recall the  Bernstein-type inequality in 
\prettyref{lmm:bernstein} for weight $w = \frac{\varphi^2}{f}$, where $f$ is 
Gaussian convolution of a compactly supported measure.
In this section
we significantly extend this result and prove \prettyref{prop:exp-bernstein} for distributions with exponential tails as defined by \prettyref{eq:Galpha}.

We first state a more general result.     Let
    \[
      w(y)=e^{-V(y)}
    \]
    be a strictly positive weight function on $\mathbb R$, with finite moments of all orders.

    \begin{lemma}\label{lem:diff-op-bound}
      Assume that
      \begin{equation}\label{eq:Vp-linear-bound}
        |V'(y)|\le C_1(1+|y|),\qquad y\in\mathbb R,
      \end{equation}
      and
      \begin{equation}\label{eq:Vpp-lower-bound}
        V''(y)\ge C_2,\qquad y\in\mathbb R,
      \end{equation}
      for some constants $C_1,C_2>0$.

      Then, for every integer $k\ge 1$ and every polynomial $p$ of degree at most $k$,
      \[
        \|p'\|_{L_2(w)}
        \le
        C\,\sqrt{k}\,\log (k+1)\,\|p\|_{L_2(w)},
      \]
      where $C>0$ depends only on $C_1,C_2$.
    \end{lemma}

    \begin{proof}
Throughout the proof, the inner product between $f,g\in L_2(w)$ is
$\langle f,g\rangle := \int_{\reals} fg w$. 
By assumption \prettyref{eq:Vpp-lower-bound}, $w$ is upper bounded by a Gaussian tail so we have the integration by parts identity \prettyref{eq:adjoint} for any $f,g$ that grow at most exponentially fast.
    
      Let $\{q_j\}_{j\ge 0}$ be the orthonormal polynomials in $L_2(w)$, so that $
        \langle q_i,q_j\rangle=\delta_{ij}$.
      For $k\ge 0$, let $\mathcal P_k$ denote the space of polynomials of degree at most $k$, and represent the differentiation operator on $\mathcal P_k$ in the basis $\{q_0,\dots,q_k\}$:
      \[
        L:=\bigl(\langle q_i,q_j'\rangle\bigr)_{0\le i,j\le k}.
      \]
      Since $q_j'\in \mathcal P_{j-1}$, the matrix $L$ is strictly upper triangular.

      Next consider the symmetric matrix
      \[
        S:=
        \bigl(\langle q_i,V' q_j\rangle\bigr)_{0\le i,j\le k}.
      \]

      By integration by parts, for $i<j$,
      \[
        \langle q_i,q_j'\rangle
        =
        \langle q_i,V' q_j\rangle.
      \]
      Also, $\langle q_i,V' q_i\rangle
        =
        2\langle q_i',q_i\rangle
        =0$. Hence, $S$ has zero diagonal and satisfies $S = L+L^\top$.
In other words, $L$ is the strictly upper triangular part of $S$. Therefore, there exists a universal constant $C_0$ so that
\cite{Bhatia2000}:
      \[
        \|L\|_{\op}
        \le
        C_0 \log(k+1)\,\|S\|_{\op},
      \]

      It remains to bound $\|S\|_{\op}$. Note that  $S$ is the matrix of the  operator
      $\Pi_k M_{V'}:\mathcal P_k\to \mathcal P_k$ of first multiplying by $V'$ then projecting to $\calP_k$, namely, 
      $\Pi_k M_{V'}(p) = \Pi_k (pV')$, where
      $\Pi_k$ denotes the orthogonal projection onto $\mathcal P_k$. Then
      \[
        \|S\|_{\op}
        \le
        \sup_{0\ne p\in \mathcal P_k}\frac{\|V' p\|_{L_2(w)}}{\|p\|_{L_2(w)}}.
      \]
      Using \eqref{eq:Vp-linear-bound},
      \[
        \|V' p\|_{L_2(w)}
        \le
        C_1\bigl(\|p\|_{L_2(w)}+\|y p\|_{L_2(w)}\bigr).
      \]
      Thus it suffices to bound the operator norm of multiplication by $y$ on $\mathcal P_k$.

      We now claim that
      \begin{equation}\label{eq:y-mult-bound}
        \|y p\|_{L_2(w)}
        \le
        C'\sqrt{k}\,\|p\|_{L_2(w)},
        \qquad p\in \mathcal P_k,
      \end{equation}
      for a constant $C'$ depending only on $C_1,C_2$. Once this is proved, we get
      \[
        \|S\|_{\op}
        \le
        C_1\bigl(1+C'\sqrt{k}\bigr)
        \le
        C''\sqrt{k},
      \]
      hence
      \[
        \|L\|_{\op}
        \le
        C_0 C'' \sqrt{k}\log(k+1),
      \]
      which is the desired estimate.

      It remains to prove \eqref{eq:y-mult-bound}. Recall that any orthogonal polynomial satisfies a three-term recurrence:\footnote{
Indeed, expanding the degree-$(j+1)$ polynomial $y q_j$ in the orthonormal basis $\{q_0,\dots,q_{j+1}\}$, only the last three terms $q_{j-1}$, $q_{j+1}$, and $q_j$ survive, because for $i\le j-2$, we have $\langle y q_j,q_i\rangle=\langle q_j,y q_i\rangle=0$.}
\begin{equation}
y q_j
        =
        \alpha_{j+1}q_{j+1}+\beta_j q_j+\alpha_j q_{j-1}.
    \label{eq:threeterm}
\end{equation}
In other words, 
in the orthonormal basis $\{q_0,\dots,q_{k+1}\}$, multiplication by $y$ from $\mathcal P_k$ to $\mathcal P_{k+1}$ is represented by the first $(k+1)$ columns of the following $(k+2)\times (k+2)$ symmetric tridiagonal matrix (known as the Jacobi matrix):
      \[J:=
        \begin{pmatrix}
          \beta_0 & \alpha_1 & 0 & 0 & \cdots & 0 & 0\\
          \alpha_1 & \beta_1 & \alpha_2 & 0 & \cdots & 0 & 0\\
          0 & \alpha_2 & \beta_2 & \alpha_3 & \cdots & 0 & 0\\
          0 & 0 & \alpha_3 & \beta_3 & \ddots & \vdots & \vdots\\
          \vdots & \vdots & \vdots & \ddots & \ddots & \alpha_k & 0\\
          0 & 0 & 0 & \cdots & \alpha_k & \beta_k & \alpha_{k+1}\\
          0 & 0 & 0 & \cdots & 0 & \alpha_{k+1} & \beta_{k+1}
        \end{pmatrix}.
      \]
      As a result,
         \[
        {\sup_{0\ne p\in \mathcal P_k}\frac{\|y p\|_{L_2(w)}}{\|p\|_{L_2(w)}}
        \le
        \|J\|_{\op}.}
      \]

      We next bound the three-term recurrence coefficients in \prettyref{eq:threeterm}. By \eqref{eq:Vpp-lower-bound}, for every $y\in\mathbb R$,
      \[
        yV'(y)\ge C_2 y^2-|V'(0)|\,|y|
        \geq C_2y^2-C_1|y|.
      \]
      Indeed, for $y>0$,
      \[
        V'(y)=V'(0)+\int_0^y V''(t)\,dt \ge V'(0)+C_2 y,
      \]
      while for $y<0$,
      \[
        V'(y)=V'(0)-\int_y^0 V''(t)\,dt \le V'(0)+C_2 y,
      \]
      and both cases yield the displayed inequality after multiplying by $y$.

      Now fix $j\ge 0$. 
      Integration by parts gives
      \[
  \int yV' q_j^2 w
  = - \int yq_j^2 w'
  = \int q_j^2 w + 2 \int y q_j q_j' w
  = 1 + 2j
      \]
      where we used the fact that 
      $y q_j' = j q_j + p_{j-1}$ for some polynomial $p_{j-1}$ of degree at most $j-1$.
      Combining this with the preceding lower bound on $yV'(y)$, we get
      \[
        2j+1
        \ge
        \int  q_j^2 w(C_2y^2-C_1|y|)
        \geq 
        \frac{C_2}{2}\int y^2 q_j^2 w-\frac{C_1^2}{2C_2},
      \]
      hence
      \[
        \int y^2 q_j(y)^2 w(y)\,dy
        \le
        \frac{4j+2}{C_2}+\frac{C_1^2}{C_2^2}.
      \]
      By the three-term recurrence and orthonormality,
      \[
        \alpha_{j+1}^2+\beta_j^2+\alpha_j^2
        =
        \int y^2 q_j^2 w
        \le
        \frac{4j+2}{C_2}+\frac{C_1^2}{C_2^2}.
      \]
      Thus
      \[
        \|J\|_{\op}
        \le
       \max_i \sum_j |J_{ij}|
        \le
        {\sup_{0\le j\le k+1}\bigl(|\alpha_j|+|\beta_j|+|\alpha_{j+1}|\bigr)}
        \le
        {\sqrt{3\left(\frac{4k+6}{C_2}+\frac{C_1^2}{C_2^2}\right)}.}
      \]
      This proves \eqref{eq:y-mult-bound}, and completes the proof.
    \end{proof}

    Next, we verify the conditions on the score function $V'$ in the preceding lemma for the special case of weight
    \begin{equation}
      w(y)=e^{-V(y)}=\frac{\varphi(y)^2}{f(y)},
      \qquad
      f=\nu * \varphi,
      \label{eq:w-general}
    \end{equation}
    where $\varphi(y)=\frac{1}{\sqrt{2\pi}} e^{-y^2/2}$ is the standard Gaussian density and $\nu$ is a probability measure on $\reals$.

\begin{lemma}[Properties of the weight \prettyref{eq:w-general}]
\label{lem:Vpp-phi2-over-f}
\begin{enumerate}
    \item For any $\nu$,
\begin{equation}
V''(y)\ge 1,\qquad y\in\mathbb R.
\label{eq:V2}
\end{equation}

\item Assume that
\begin{equation}
\nu([-a,a]) \ge \frac12.
\label{eq:nubulk}
\end{equation}
Then
\begin{equation}
|V'(y)| \le \bigl(3+a+\sqrt{\log 4}\bigr)(|y|+1),
\qquad y\in\mathbb R.
\label{eq:V1}
\end{equation}
\end{enumerate}
\end{lemma}

\begin{proof}
We first prove \prettyref{eq:V2}.
Let $U\sim \nu$, $Z\sim N(0,1)$, and $Y=U+Z$, so that $Y$ has density $f$. Define
\[
m(y):=\E[U\mid Y=y].
\]
Then
\[
\frac{f'(y)}{f(y)}=m(y)-y,
\]
hence
\begin{equation}
    V'(y)=-(\log w)'(y)=2y+\frac{f'(y)}{f(y)}=y+m(y).
    \label{eq:Vp-upper}
\end{equation}
Differentiating once more,
\[
V''(y)=1+m'(y).
\]
For Gaussian noise, the posterior mean satisfies
\[
m'(y)=\Var(U\mid Y=y)\ge 0.
\]
Therefore
\[
V''(y)=1+\Var(U\mid Y=y)\ge 1.
\]

Next, we prove \prettyref{eq:V1}. Recall an elementary bound on the score from \cite[Lemma A.1]{JiangZhang2009}:
\begin{equation}
\left(\frac{f'(y)}{f(y)}\right)^2
\le
\log\frac{1}{2\pi f(y)^2}.
\label{eq:JZ}
\end{equation}
Indeed, with $U\sim \nu$ and $Y=U+Z$, applying Jensen's inequality yields
\[
\pth{\frac{f'(y)}{f(y)}}^2
=(\Expect[U-y|Y=y])^2
\leq \Expect[(U-y)^2|Y=y]
\leq 2 \log \Expect[\exp((U-y)^2/2)|Y=y]
= \log\frac{1}{2\pi f(y)^2}.
\]

Note that $f(y)$ is lower bounded using  \prettyref{eq:nubulk} as follows:
\[
    f(y)
=
\int \varphi(y-u)\,\nu(du)
\ge
\int_{[-a,a]} \varphi(y-u)\,\nu(du)
\ge
\frac12 \inf_{|u|\le a}\varphi(y-u)
\geq \frac{1}{2\sqrt{2\pi}}e^{-(|y|+a)^2/2}.
\]
Substituting this into \prettyref{eq:JZ},
\[
\left|\frac{f'(y)}{f(y)}\right|
\le
\sqrt{(|y|+a)^2+\log 4}
\le
|y|+a+\sqrt{\log 4}.
\]
Combining this with \prettyref{eq:Vp-upper} completes the proof.
\end{proof}

Finally, we prove \prettyref{prop:exp-bernstein}. By assumption, $\nu = (G+H)/2 \in \calG_\alpha(\sigma)$. Then
$\nu([-a,a])\geq \frac12$ with $a:=\sigma (\log 4)^{1/\alpha}$ by Markov's inequality. Then \prettyref{prop:exp-bernstein} follows from \prettyref{lem:diff-op-bound}, with conditions verified by \prettyref{lem:Vpp-phi2-over-f}.

\section{Counterexamples for moment class}
\label{app:moment}
Recall that \prettyref{thm:ub-compact-intro} and \prettyref{thm:ub-exp} bound the regret proportionally to the squared Hellinger distance up to logarithmic factors for priors with compact support or exponential tails.
In this appendix, we show that such a
nearly linear regret bound fails for the moment class.

\begin{prop}\label{prop:moment-counterexample}
Fix $p > 1$. There exists $(G_m,H_m)$ with uniformly bounded $p$th moments such that
$\epsilon_m^2 = \hellinger^2 (f_ {G_m},f_ {H_m}) \to 0$ and
$
\Regret(H_m\|G_m) \geq c \epsilon_m^{2-2/p}
$
for some constant $c$.
\end{prop}

\begin{proof}
For $b>0$, let
\[H=\delta_0,
\qquad
G=(1-\eta)\delta_0+\eta\delta_b,
\qquad
\eta=b^{-p}
\]
both with $p$th moment at most 1.

Note that $f_H(y)=\varphi(y)$, and $
f_G(y)=(1-\eta)\varphi(y)+\eta \varphi(y-b)$.
By data processing inequality,
\[
\hellinger^2(f_G,f_H)
\le \hellinger^2(\Bern(\eta),\Bern(0))
\le 2\eta.
\]

To lower bound the regret,
note that the Bayes estimator under $H$ is $m_H(y)=0$ so
\[
\Regret(H\|G)
= \Expect[\Expect[U|Y]^2]
\ge \var(U) - \Expect[(U-\Expect[U|Y])^2],
\]
where $U\sim G\indep Z\sim N(0,1)$, and $Y=U+Z$.
Write $U=b W$ with $W \sim \Bern(\eta)$.
Then $\Var(U)=\eta(1-\eta)b^2$. To bound the conditional variance from above, consider a suboptimal estimator $\hat U= b \hat W$ and $\hat W = \indc{Y\geq b/2}$.
Then
$\Prob[W\neq \hat W] \leq \prob{|Z| \geq b/2} \leq \exp(-b^2/8)$ and hence
$\Expect[(U-\Expect[U|Y])^2]
\leq \Expect[(U-\hat U)^2]
\leq  b^2 \exp(-b^2/8)$.
In all,
\[
\Regret(H\|G) \geq b^2(\eta(1-\eta) - \exp(-b^2/8)).
\]
\end{proof}

\section{Proof of \prettyref{cor:NPMLE}}
\label{app:NPMLE}

As mentioned after \prettyref{cor:NPMLE}, the new regret-Hellinger inequalities are  readily applicable to constrained NPMLE. With a modicum of extra effort, they also apply to unconstrained NPMLE, again without regularization. Let us consider the  subgaussian class as an example. Recall the unconstrained NPMLE \eqref{eq:NPMLE}.
It is known that, deterministically,
\begin{itemize}
    \item $\hat G$ exists uniquely and is a discrete measure with at most $n$ atoms supported on the sample range $[y_{\min},y_{\max}]$ \cite{lindsay1983geometry1}.

    \item Suppose that $y_1,\ldots,y_n$ are drawn independently from $f_G$, where the true
    prior 
    $G \in \calG_2(\sigma)$ is subgaussian with proxy variance $\sigma^2$. Then $\hellinger^2(f_
    {\hat
    G},f_G) \leq \frac{C (\log n)^2}{n}$ with probability at least $1-n^{-c}$ for some constants $c,C$ \cite{Zhang_2009}.
\end{itemize}

We claim that on the high-probability event $A$ that $\hellinger^2(f_{\hat G},f_G) \leq
\frac{C(\log n)^2}{n}$ and $\hat G$ is supported on $[-\sqrt{C\log n},\sqrt{C\log n}]$, $\hat
 G$ is subgaussian with proxy variance at most a constant. Since $A$ can be made to have
 probability at least $1-n^{-2}$ and regret is bounded by $2(\max_i |y_i|)^2 + C_\sigma$,
 one could use Cauchy--Schwarz to bound $\E_G[\Regret(\hat G \Vert G) \one(A^c)]$. Thus
 it suffices to bound regret on the event $A$.
  Applying \prettyref
 {thm:ub-exp} with $\alpha=2$ and $\epsilon^2 = \frac{C(\log n)^2}{n}$, we obtain
 $\Regret(\hat G\|G) = O(\frac{1}{n} (\log n)^3 (\log \log n)^2)$.
 The claim on the proxy variance of $\hat G$ is justified by applying the following lemma
 and $\TV \leq \hellinger$. 

\begin{lemma}
\label{lem:tv-subgaussian-transfer}
Let $\varphi$ denote the standard normal density, and $
f_G = G * \varphi$, $f_H = H * \varphi$.
Assume that:
\begin{enumerate}
    \item $H$ is supported on $[-A,A]$ for some $A>0$.
    \item $G$ is $C$-subgaussian, in the sense that if $V \sim G$, then
    \[
    \mathbb E e^{\lambda V}
    \le
    \exp\!\left(\frac{C^2\lambda^2}{2}\right)
    \qquad
    \text{for all } \lambda \in \mathbb R.
    \]
    \item For some $c>0$,
    \[
    \TV(f_G,f_H)\le e^{-cA^2}.
    \]
\end{enumerate}
Then,
$H$ is $C'$-subgaussian, with
\[
C' \le C_0 \max\!\left\{C,\frac{1}{\sqrt c}\right\}
\]
for some universal constant $C_0$.
\end{lemma}

\begin{proof}
Let
\[
V \sim G,
\qquad
U \sim H,
\qquad
Z \sim N(0,1),
\]
with $V,U,Z$ independent, and define
\[
Y_G := V+Z,
\qquad
Y_H := U+Z.
\]
Then $Y_G$ has density $f_G$, and $Y_H$ has density $f_H$.

Since $V$ is $C$-subgaussian,
$Y_G$ is $\sqrt{C^2+1}$-subgaussian. By the Chernoff bound,
\[
\mathbb P(|Y_G|>s)
\le
2\exp\!\left(-\frac{s^2}{2(C^2+1)}\right)
\qquad
\text{for all } s\ge 0.
\]
Let $\delta := \TV(f_G,f_H)$. By the definition of total variation distance,
\[
\mathbb P(|Y_H|>s)
\le
2\exp\!\left(-\frac{s^2}{2(C^2+1)}\right)+\delta.
\]

Next we pass from $Y_H=U+Z$ back to $U$. If $|U|>t$ and $|Z|\le t/2$, then $
|Y_H| = |U+Z| > t/2$. Therefore
\[
\mathbb P(|U|>t)
\le
\mathbb P(|Y_H|>t/2)+\mathbb P(|Z|>t/2)
\le
2\exp\!\left(-\frac{t^2}{8(C^2+1)}\right)
+
\delta
+
2e^{-t^2/8}.
\]
Since $H$ is supported on $[-A,A]$, we only consider $0\le t\le A$. In this range, assumption (3) implies
$\delta
\le
e^{-cA^2}
\le
e^{-ct^2}$.
Hence,
\[
\mathbb P(|U|>t)
\le
2\exp\!\left(-\frac{t^2}{8(C^2+1)}\right)
+
\exp\!\left(-ct^2\right)
+
2e^{-t^2/8}
\leq 5\exp\!\left(-\frac{t^2}{L^2}\right)
\]
with $L^2 := \max\!\left\{8(C^2+1), \frac{1}{c}\right\}$.
\end{proof}

\newcommand{\eps}{\epsilon}
\newcommand{\Hell}{\hellinger}
\newcommand{\coverN}{N}
\newcommand{\bracketN}{N_{[]}}

\section{Optimal estimation of Gaussian mixtures with compactly supported priors}
\label{sec:optimal-density}

In this appendix we present two results on density estimation in terms of Hellinger distance. Let
\[y_1,\ldots, y_n \simiid 
f_G = \varphi * G, \quad G\in\mathcal P([-M,M]).
\]
be independently drawn from a Gaussian mixture density with a compactly supported prior $G$. 
We use $\Expect_{G}$ to denote expectation taken over this sample.\footnote{For simplicity, we present the results in terms of expected Hellinger risk. These rates hold typically, as the existing results on Hellinger risk through local Hellinger covering and bracketing entropy hold with high probability.}

The following theorem resolves the optimal Hellinger rate for estimating Gaussian mixtures for compactly supported priors, matching the minimax lower bound in
\cite[Theorem 20]{polyanskiy2021sharp} and improving the known best upper bound of $O(\frac{1}{n}(\frac{\log n}{\log\log n})^{1.5})$ \cite{NW21}.

\begin{theorem}[Optimal Hellinger rate for Gaussian mixtures]
\label{thm:lcb-risk}
For any $M>0$, there is a constant $C_M$ such that the following holds. 
Given $(y_1,\ldots, y_n) \simiid f_G$, there is an estimator $\tilde G \in
\mathcal P([-M, M])$, such that, for all $n\ge3$,
\[
\sup_{G\in\calP([-M,M])}
\mathbb E_{G}\Hell^2(f_{\tilde G},f_G)
\le
C_M\frac{\log n}{n\log\log n}.
\]
\end{theorem}

\prettyref{thm:lcb-risk} applies the Birg\'e--Le Cam estimator \cite{Birge83}, which is computationally intractable. Next we show that the constrained NPMLE attains the optimal rate up to a $\log\log n$ factor.

\begin{theorem}[Constrained NPMLE]
\label{thm:npmle-risk}
Fix $M<\infty$, and let
\[
\hat G_n=\argmax_{H \in \mathcal P([-M,M])}\sum_{i=1}^n\log f_H(y_i)
\]
be the support-constrained NPMLE.  Then, for all $n\ge3$,
\[
\sup_{G\in\calP([-M,M])}
\mathbb E_{G} \Hell^2(f_{\hat G_n},f_G)
\le
C_M\frac{\log n}{n}.
\]
\end{theorem}

Coupled with the Regret-Hellinger inequality in \prettyref{thm:ub-compact-intro}, 
Theorems \ref{thm:lcb-risk}--\ref{thm:npmle-risk} immediately lead to the regret bounds in
Corollaries \ref{cor:lcb-regret}--\ref{cor:npmle-regret} in \prettyref{sec:statsapp}.
Both theorems are direct consequences of new results on the \emph{local} Hellinger covering and bracketing entropy for Gaussian mixtures,\footnote{In fact, for Gaussian mixtures whose mixing distributions are  compactly supported, it is known that the optimal (minimax) Hellinger
rate is \textit{determined} by the local Hellinger entropy \cite{jia.polyanskiy.wu.2023}.} which are given in \prettyref{sec:local-covering-number} and \prettyref{sec:local-bracketing-number}, respectively. For completeness, \prettyref{sec:remaining-proof} provides the proof of Theorems \ref{thm:lcb-risk}--\ref{thm:npmle-risk} and Corollaries \ref{cor:lcb-regret}--\ref{cor:npmle-regret}.

\subsection{Local Hellinger entropy}
\label{sec:local-covering-number}
For a metric
space $(\mathcal A,d)$, write $\coverN(r,\mathcal F,d)$ for the minimal cardinality of an
$r$-covering for $\calF\subset\calA$, that is, the minimum number of balls of radius $r$ whose union covers $\calF$. 
Denote the collection of Gaussian mixtures
\[
\calF_M:=\{f_G:G\in\mathcal P([-M,M])\}.
\]
 The local Hellinger covering number of this model class is
\[
\coverN_{\mathrm{loc},H}(\calF_M,\eps)
:=
\sup_{f_0\in\calF_M,\ \delta\ge\eps}
\coverN\left(
\delta/2,\,
\calF_M\cap\{f:\Hell(f,f_0)\le\delta\},\,
\Hell
\right).
\]

The following result finds the optimal rate for this local entropy. 

\begin{theorem}[Local Hellinger entropy for compactly supported priors]
\label{thm:localcovering}
Fix $M<\infty$.  There is a constant $C_M<\infty$ such that, for every
$f_0\in\calF_M$ and all $0<\rho\le\delta\le 0.001$,
\[
\log
\coverN\left(
\rho,\,
\calF_M\cap\{f:\Hell(f,f_0)\le\delta\},\,
\Hell
\right)
\le
C_M
\frac{\Lambda_{\rho,\delta}}{\log\Lambda_{\rho,\delta}}
\log\left(e+\frac{\delta^2}{\rho^2}\right),
\qquad
\Lambda_{\rho,\delta}
=
\log\left(e^e\frac{\delta}{\rho^2}\right).
\]
Consequently, for all sufficiently small $0<\eps<e^{-e}$,
\begin{equation}
\log \coverN_{\mathrm{loc},H}(\calF_M,\eps)
\le
C_M\frac{\log(1/\eps)}{\log\log(1/\eps)}.
\label{eq:localcovering}
\end{equation}
\end{theorem}

To prove this result, we first record an elementary lemma.

\begin{lemma}[Local likelihood-ratio comparison]
\label{lmm:l2ratiohellinger}
Let $f_0$ be a density and $0\le R<\infty$.  
If $0<\rho\le\delta$, $a,b\ge0$, $\int a f_0=\int b f_0=1$, and
\[
\|a-1\|_{L_2(f_0)}\le R\delta,
\qquad
\|b-1\|_{L_2(f_0)}\le R\delta,
\qquad
\|a-b\|_{L_2(f_0)}
\le
\frac{1}{8(1+R)}
\frac{\rho^2}{\delta},
\]
then
\begin{equation}
\label{eq:l2-ratio-to-hellinger}
\int_{\mathbb R}(\sqrt a-\sqrt b)^2f_0\le\rho^2.
\end{equation}
\end{lemma}
\begin{proof}[Proof of \prettyref{lmm:l2ratiohellinger}]
Define
$\gamma:=\frac{\rho^2}{16(1+R)^2\delta^2}$.
  On
$A=\{a+b\ge\gamma\}$,
\[
(\sqrt a-\sqrt b)^2\le\frac{(a-b)^2}{a+b}
\le \gamma^{-1}(a-b)^2.
\]
Consequently,
\[
\begin{aligned}
\int_A(\sqrt a-\sqrt b)^2f_0
\le
\gamma^{-1}\|a-b\|_{L_2(f_0)}^2 
\le
\frac{1}{\gamma (8(1+R))^2}
\frac{\rho^4}{\delta^2}
=\frac{\rho^2}{4}.
\end{aligned}
\]

On $A^c$, since $a<\gamma$ and $b<\gamma$, we have
$|a-1|\ge 1-\gamma$ and $|b-1|\ge 1-\gamma$. 
Hence
\[
\mathbf 1_{A^c}a\le
\frac{\gamma}{(1-\gamma)^2}\mathbf 1_{A^c}(a-1)^2,
\qquad
\mathbf 1_{A^c}b\le
\frac{\gamma}{(1-\gamma)^2}\mathbf 1_{A^c}(b-1)^2.
\]
Therefore
\[
\int_{A^c}(\sqrt a-\sqrt b)^2\,f_0
\le \int_{A^c}(a+b)\,f_0
\le \frac{2\gamma R^2}{(1-\gamma)^2}\delta^2
\leq 4\gamma R^2\delta^2
=
\frac{R^2}{4(1+R)^2}\rho^2
\le
\frac{\rho^2}{4},
\]
where we used $\gamma\le 1/16$ because $\rho\le\delta$.
Adding the contributions over $A$ and $A^c$ proves
\eqref{eq:l2-ratio-to-hellinger}.
\end{proof}

 Next, we outline the proof of the local Hellinger entropy in \prettyref{thm:localcovering}. 
The key idea is to reduce the infinite-dimensional Hellinger ball to a finite-dimensional ball whose dimension scales as the right side of \prettyref{eq:localcovering}. 
Compared with the local entropy for finite Gaussian mixtures (but in high dimensions) in \cite[Lemma 4.1 and Theorem 4.2]{doss2020optimal}, obtained by showing the Hellinger distance is comparable to the Hilbert-Schmidt norm between moment tensors of the mixing distributions, 
we do not prove such a characterization explicitly. Instead, we apply polynomial approximation to reduce a local Hellinger ball to a finite-dimensional ball and convert the covering therein back to a Hellinger covering. This results in a better (and optimal) entropy bound than the prior result in \cite{NW21} using covering in the moment sequence space. 

The argument proceeds in three steps:
\begin{enumerate}[wide]
\item \emph{Localize the likelihood ratios in $L_2$.}
While squared Hellinger distance is not comparable to $\chi^2$-divergence in general, 
\cite{jia.polyanskiy.wu.2023} shows this is in fact the case for Gaussian mixtures with compactly supported mixing distributions.
Specifically, 
\cite[Theorem~21]{jia.polyanskiy.wu.2023} shows the existence of some constant $A_M<\infty$ such
that, for all $f,g\in\calF_M$,
\begin{equation}
\label{eq:compact-chi2-hellinger}
\chi^2(f\|g)
=
\int_{\mathbb R}\frac{(f-g)^2}{g}
=
\left\|f/g-1\right\|_{L_2(g)}^2
\le
A_M^2\Hell^2(f,g).
\end{equation}
Consequently,
\[
\Hell(f,f_0)\le\delta
\quad\Longrightarrow\quad
\|r_f-1\|_{L_2(f_0)}\le A_M\delta, \qquad r_f \equiv f/f_0.
\]

\item \emph{Pass to a finite-dimensional ball.}
Polynomial approximation of $f/\varphi$ yields an approximation of $r_f-1$ by
an element $w_f$ of the $(m+1)$-dimensional space
\begin{equation}
V_m(f_0)
=
\left\{
x\mapsto\frac{\varphi(x)p(x)}{f_0(x)}
:
p\text{ is a polynomial of degree at most }m
\right\}.
\label{eq:Vm}
\end{equation}
We choose $m = \Theta_M(\frac{\log (\delta/\rho^2)}{\log\log(\delta/\rho^2)})$ with $\rho
= \delta/2$ so that, by moment matching we have
\[
\|r_f-1-w_f\|_{L_2(f_0)}
=O_M\left(\frac{\rho^2}{\delta}\right).
\]
Since $\rho\le\delta$, the first step then gives
\[
\|w_f\|_{L_2(f_0)}=O_M(\delta).
\]
Thus the local Hellinger ball is represented, up to an
$O_M(\rho^2/\delta)$ error, inside a ball of radius $O_M(\delta)$ in
$V_m(f_0)$.

\item \emph{From finite-dimensional covering back to Hellinger covering.}
Applying the volumetric bound, we obtain an $L_2(f_0)$-covering of the preceding finite-dimensional ball, denoted by $\{w_{f_j}: j=1,\ldots,N\}$, such that 
$\log N = O_M(m)$ and any $w_f$ is covered by some $w_{f_j}$ such that
\[
\|w_f-w_{f_j}\|_{L_2(f_0)}
=O_M\left(\frac{\rho^2}{\delta}\right).
\]
Adding the two approximation errors gives
\[
\|r_f-r_{f_j}\|_{L_2(f_0)}
=O_M\left(\frac{\rho^2}{\delta}\right).
\]
Both $r_f$ and $r_{f_j}$ are
$O_M(\delta)$-close to $1$ by the first step, so
\prettyref{lmm:l2ratiohellinger} implies $\Hell(f,f_j)\le\rho$.  Hence this
finite-dimensional covering yields the desired Hellinger $\rho$-covering.
\end{enumerate}

\begin{proof}[Proof of \prettyref{thm:localcovering}]
Throughout the proof all constants depend only on $M$ and may change from line
to line.  

Fix a center $f_0=f_{G_0}\in\calF_M$ and set
$r_f\equiv f/f_0$.  By shifting and replacing $M$ by $2M$, we may assume
without loss of generality that $G_0$ has zero mean. For each integer $m\ge1$,
define the $(m+1)$-dimensional linear space $V_m(f_0)$ in \prettyref{eq:Vm}.
We first quantify the approximation accuracy of $V_m(f_0)$ for $f/f_0$: there are
constants $B_M,C_M<\infty$, independent of $f_0$, such that
\begin{equation}
\label{eq:ratio-polynomial-approx}
\sup_{f\in\calF_M}\inf_{v\in V_m(f_0)}
\|r_f-v\|_{L_2(f_0)}
\le B_M\left(\frac{C_M}{m}\right)^{m/2}.
\end{equation}
Indeed, by Jensen's inequality,
\[
\frac{f_0(x)}{\varphi(x)}
=\int e^{ux-u^2/2}\,G_0(du)
\ge
\exp\left\{-\frac12\int u^2\,G_0(du)\right\}
\ge c_M.
\]
This shows $V_m(f_0) \subset L_2(f_0)$.
In addition, for every polynomial $p$ of degree
at most $m$,
\[
\left\|r_f-\frac{\varphi p}{f_0}\right\|_{L_2(f_0)}^2
=
\int_{\mathbb R}\left(\frac{f}{\varphi}-p\right)^2
\frac{\varphi^2}{f_0}\,dx
\le
C_M\int_{\mathbb R}\left(\frac{f}{\varphi}-p\right)^2\varphi\,dx .
\]
Let $H_j$ denote the Hermite polynomials, satisfying
\[
e^{ux-u^2/2}
=\sum_{j=0}^{\infty}\frac{u^j}{j!}H_j(x),
\qquad
\int_{\mathbb R}H_i(x) H_j(x)\varphi(x)\,dx=j! \mathbf{1}_{i=j}.
\]
Then
\begin{equation}
\frac{f_G}{\varphi}
=
\sum_{j=0}^{\infty}
\frac{m_j(G)}{j!}H_j,
\qquad
m_j(G)=\int u^j\,G(du).
\label{eq:hermite-exp}
\end{equation}
Taking
\[
p_m=\sum_{j=0}^m\frac{m_j(G)}{j!}H_j
\]
and using $|m_j(G)| \leq (2M)^j$, we conclude \eqref{eq:ratio-polynomial-approx} from
\[
\int_{\mathbb R}\left(\frac{f_G}{\varphi}-p_m\right)^2\varphi\,dx
=
\sum_{j>m}\frac{m_j(G)^2}{j!}
\le
\sum_{j>m}\frac{(2M)^{2j}}{j!}
\le
B_M\left(\frac{C_M}{m}\right)^m.
\]

Now fix $0<\rho\le\delta\le2$ and let
$
\mathcal B(f_0,\delta)=\{f\in\calF_M:\Hell(f,f_0)\le\delta\}.
$
By \eqref{eq:compact-chi2-hellinger}, every $f\in\mathcal B(f_0,\delta)$
satisfies
\begin{equation}
\label{eq:local-ratio-radius}
\|r_f-1\|_{L_2(f_0)}\le A_M\delta.
\end{equation}
Let $R=2A_M+1$, and let $\eta=\frac{1}{8(1+R)}$.
Choose
\[
m=m_{\rho,\delta}
=\left\lceil D_M
\frac{\Lambda_{\rho,\delta}}{\log\Lambda_{\rho,\delta}}\right\rceil
\]
with $D_M$ large enough so that the right side of
\eqref{eq:ratio-polynomial-approx} is at most
$\eta\rho^2/(16\delta)$.  For each $f\in\mathcal B(f_0,\delta)$, apply
\eqref{eq:ratio-polynomial-approx} to both $f$ and $f_0$.  Thus there are
$v_f,v_0\in V_m(f_0)$, with $v_0$ depending only on $f_0$, such that
\[
\|r_f-v_f\|_{L_2(f_0)}\le \eta\frac{\rho^2}{16\delta},
\qquad
\|1-v_0\|_{L_2(f_0)}\le \eta\frac{\rho^2}{16\delta}.
\]
Since $V_m(f_0)$ is linear and $r_{f_0}=1$, the function $w_f\equiv v_f-v_0$ belongs
to $V_m(f_0)$ and satisfies
\begin{equation}
\label{eq:centered-ratio-approx}
\|r_f-1-w_f\|_{L_2(f_0)}\le \eta\frac{\rho^2}{8\delta}.
\end{equation}

Combining \eqref{eq:local-ratio-radius} and
\eqref{eq:centered-ratio-approx}, all $w_f$'s lie in the $L_2(f_0)$-ball of
radius $(A_M+\eta/8)\delta$ inside the $(m+1)$-dimensional space $V_m(f_0)$.
First cover this ball at radius $\eta\rho^2/(4\delta)$, 
then choose one element of $\{w_f:f\in\mathcal B(f_0,\delta)\}$ from each
covering ball that intersects this set.  This gives a proper
$\eta\rho^2/(2\delta)$-covering of $\{w_f:f\in\mathcal B(f_0,\delta)\}$ in
$L_2(f_0)$-norm.
Denote this cover by $\{w_{f_j}:1\le j\le N\}$, where 
$f_1,\ldots,f_N\in\mathcal B(f_0,\delta)$ and, by the Euclidean volume bound\footnote{Specifically, for $a\in\reals^{m+1}$, denote by $p_a$ the degree-$m$ polynomial with coefficients given by $a$. In view of the definition of $V_m(f_0)$ in \prettyref{eq:Vm}, the volumetric bound is applied to the norm ball in $\reals^{m+1}$ defined by $\|a\| := \|p_a \varphi/f_0\|_{L_2(f_0)}$, which is a weighted Euclidean norm.}
(cf.~e.g.~\cite[Corollary 27.4]{polyanskiy.wu.2025}),
\[
N\le
\left(1+\frac{C_M\delta^2}{\eta\rho^2}\right)^{m+1}.
\]
For every $f\in\mathcal B(f_0,\delta)$, choose $j$ such that
$\|w_f-w_{f_j}\|_{L_2(f_0)}\le\eta\rho^2/(2\delta)$.  Then
\[
\begin{aligned}
\|r_f-r_{f_j}\|_{L_2(f_0)}
&\le
\|r_f-1-w_f\|_{L_2(f_0)}
+\|w_f-w_{f_j}\|_{L_2(f_0)}
+\|w_{f_j}-(r_{f_j}-1)\|_{L_2(f_0)}  \\
&\le \eta\frac{\rho^2}{8\delta}
+\eta\frac{\rho^2}{2\delta}
+\eta\frac{\rho^2}{8\delta}
\le \eta\frac{\rho^2}{\delta}.
\end{aligned}
\]

Finally, together with \eqref{eq:local-ratio-radius}, applying 
\eqref{eq:l2-ratio-to-hellinger} in 
\prettyref{lmm:l2ratiohellinger} to $a=r_f$ and $b=r_{f_j}$ gives
$\Hell(f,f_j)
=\left\|\sqrt{r_f}-\sqrt{r_{f_j}}\right\|_{L_2(f_0)}
\le \rho$. 
Therefore
\[
\coverN\left(\rho,
\calF_M\cap\{f:\Hell(f,f_0)\le\delta\},\Hell\right)
\le
\left(1+\frac{C_M\delta^2}{\rho^2}\right)^{m_{\rho,\delta}+1}.
\]
Taking logarithms gives the first claim.  The displayed bound for
$\coverN_{\mathrm{loc},H}$ follows by taking $\rho=\delta/2$ and then the
supremum over $f_0\in\calF_M$ and $\delta\ge\eps$.
\end{proof}

\subsection{Local bracketing entropy}
\label{sec:local-bracketing-number}

Let $\mathcal A$ be a class of densities with respect to the Lebesgue measure.
A Hellinger $\eta$-bracket is a pair of nonnegative functions $(L,U)$ with
$L\le U$ pointwise and
\[
\|\sqrt U-\sqrt L\|_{L_2(dx)}\le \eta .
\]
The Hellinger bracketing number $\bracketN(\eta,\mathcal A,\Hell)$ is the
smallest integer $N$ for which there are Hellinger $\eta$-brackets
$(L_j,U_j)_{j=1}^N$ such that, for every $f\in\mathcal A$, one has
$L_j\le f\le U_j$ pointwise for some $j$.  
Similar to the definition of covering numbers, the bracket endpoints
need not belong to the original class or be probability densities.

Clearly, bracketing is a stronger notion than covering. Nevertheless,
using an additional argument, we can upgrade the result in Section~\ref{sec:local-covering-number} from local covering to bracketing. To this end, a standard method is to bound the so-called \emph{envelope} function which controls the pointwise difference of densities. This general result is stated as follows (see \cite[Theorem 2.7.11]{wellner2013weak} for a similar result):

\begin{lemma}[Upgrading covering to bracketing via envelope]
\label{lmm:cover-to-bracket-full}
Let $\mathcal A$ be a class of densities.  Suppose that for some
$s,u>0$ there is a nonnegative function $B(\cdot)$ such that
\[
|\sqrt f-\sqrt g|\le B
\quad\text{whenever }f,g\in\mathcal A\text{ and }\Hell(f,g)\le s.
\]
If $\|B\|_2\le u/2$, then
\[
\bracketN(u,\mathcal A,\Hell)
\le
\coverN(s/2,\mathcal A,\Hell).
\]
\end{lemma}

\begin{proof}
Take an $s/2$-cover of $\mathcal A$ and choose one representative $g_j$
from each nonempty covering ball.  If $f$ lies in the same ball as $g_j$,
then $\Hell(f,g_j)\le s$, so
$(\sqrt{g_j}-B)_+^2\le f\le(\sqrt{g_j}+B)^2$. 
The Hellinger width of this bracket is at most $2\|B\|_2\le u$.
\end{proof}

The next result provides the needed envelope bound for the Gaussian mixture class.
\begin{lemma}[Square-root envelope for Gaussian mixtures]
\label{lmm:scale-dependent-full-envelope}
For every fixed $M<\infty$ there are constants $c_M,C_M,A_M<\infty$ such
that the following holds.  For $0<u\le c_M$, define
\[
L_u=\log(e^e/u),
\qquad
k_u=\left\lceil A_M\frac{L_u}{\log L_u}\right\rceil,
\qquad
s_u=c_M\frac{u}{\sqrt{k_uL_u}} .
\]
There is a nonnegative function $B_u$, depending only on $u$ and $M$, such
that, whenever $f,g\in\calF_M$ and $\Hell(f,g)\le s_u$,
\[
|\sqrt {f(x)}-\sqrt{ g(x)}|\le B_u(x)
\qquad\text{for all }x\in\mathbb R,
\qquad
\|B_u\|_{L_2(dx)}\le u/4 .
\]
\end{lemma}

With the key Lemma~\ref{lmm:scale-dependent-full-envelope}, we may bound the local bracketing number. We do so in the form of the entropy integral needed for the analysis of MLE \cite{wong.shen.1995}.
\begin{theorem}[Local bracketing integral for Gaussian mixtures]
\label{thm:full-local-bracketing-integral}
For every fixed $M<\infty$ there are constants $c_M,C_M<\infty$ such that,
for every $0<\epsilon\le c_M$,
\[
\sup_{f_0\in\calF_M}
\int_{\epsilon^2}^{\epsilon}
\sqrt{
\log
\bracketN\!\left(
u,\,
\calF_M\cap\{f:\Hell(f,f_0)\le\epsilon\},\,
\Hell
\right)
}\,du
\le
C_M\epsilon\sqrt{\log\frac1\epsilon}.
\]
\end{theorem}

\begin{proof}
Fix $f_0\in\calF_M$, $0<u\le\epsilon$, and let $s_u$ and $B_u$ be as
in Lemma~\ref{lmm:scale-dependent-full-envelope}.  Applying
Lemma~\ref{lmm:cover-to-bracket-full} to the local ball
$
\mathcal A_\epsilon(f_0)
\equiv 
\calF_M\cap\{f:\Hell(f,f_0)\le\epsilon\}$ 
and then applying Theorem~\ref{thm:localcovering}, we obtain
\[
\log
\bracketN\!\left(
u,\,
\mathcal A_\epsilon(f_0),\,
\Hell
\right)
\le
\log
N\!\left(
s_u/2,\,
\mathcal A_\epsilon(f_0),\,
\Hell
\right)
\le
C_M
\frac{\Lambda_u}{\log\Lambda_u}
\log\left(e+\frac{\epsilon^2}{s_u^2}\right),
\]
where $\Lambda_u\equiv\log\left(e^e\epsilon/s_u^2\right)$.

Let $L=\log(e^e/\epsilon)$ and write $u=\epsilon e^{-t}$, so
$0\le t\le \log(1/\epsilon)$.  Since
\[
s_u^{-2}\le C_Mu^{-2}\{\log(e^e/u)\}^2,
\]
we have
\[
\Lambda_u\leq C_M (L+t+\log(L+t)),
\qquad
\log\left(e+\frac{\epsilon^2}{s_u^2}\right)
\le C_M\{1+t+\log(L+t)\}.
\]
Since $t \leq L$, the preceding entropy bound is at most $C_M L(1+t)$, and hence
\[
\begin{aligned}
\int_{\epsilon^2}^{\epsilon}
\sqrt{
\log
\bracketN\!\left(
u,\,
\calF_M\cap\{f:\Hell(f,f_0)\le\epsilon\},\,
\Hell
\right)
}\,du
&\le
C_M\epsilon \sqrt{L} \int_0^{\log(1/\epsilon)}
e^{-t}\sqrt{1+t}\,dt \\
&\le
C_M'\epsilon\sqrt L .
\end{aligned}
\]
Taking the supremum over $f_0$ proves the theorem.
\end{proof}

It remains to prove the envelope bound in Lemma~\ref{lmm:scale-dependent-full-envelope} for nonparametric Gaussian mixture densities. 
The main idea is to first bound the envelope for finite mixtures then approximate a general mixture pointwise with a finite one. 
Recognizing that the difference between two finite Gaussian mixtures is proportional to an exponential sum, the following auxiliary result is crucial.

\begin{lemma}[Nikolskii-type inequality for exponential sums]
\label{lmm:nikolskii-exponential-sums}
There is a universal constant $A<\infty$ such that the following holds.  Let
\[
E(x)=\sum_{j=1}^m a_j e^{\lambda_j x},
\qquad a_j,\lambda_j\in\mathbb R,
\]
and let $I\subset\mathbb R$ be an interval of length $|I|$.  If $y$ is
the midpoint of $I$, then
\[
|E(y)|^2\le \frac{A m}{|I|}\int_I |E(x)|^2\,dx .
\]
\end{lemma}

\begin{proof}
\cite[Theorem~2.2]{borwein2000pointwise} states that, for
\[
\mathcal E_n
=
\left\{a_0+\sum_{j=1}^n a_j e^{\lambda_j t}:
a_j,\lambda_j\in\mathbb R\right\},
\]
there is a universal constant $c$ such that, 
for any nonzero $f\in\mathcal E_m$, any $a<y<b$ and $q>0$,
\[
|f(y)|
\le
\left(
\frac{c(1+qm)}{\min\{y-a,b-y\}}
\right)^{1/q} \|f\|_{L_q[a,b]}
\]
Taking $I=[a,b]$, $q=2$, and $y=(a+b)/2$ yields the desired bound.
\end{proof}

\begin{lemma}[Bounded Gaussian likelihood ratio]
\label{lmm:local-gaussian-comparability}
There is $C_M<\infty$ such that, for every $y\in\mathbb R$, 
$|u|\le M$ and $x\in I_y \equiv [
y-\frac{1}{8(1+|y|+M)},\,
y+\frac{1}{8(1+|y|+M)}
],$
we have
\[
C_M^{-1}\varphi(y-u)\le \varphi(x-u)\le C_M\varphi(y-u).
\]
\end{lemma}

\begin{proof}
Let $h=x-y$.  For $x\in I_y$,
$|h|\le [8(1+|y|+M)]^{-1}$.  Uniformly over $|u|\le M$,
\[
\left|\log\frac{\varphi(x-u)}{\varphi(y-u)}\right|
=
\frac12\left|(x-u)^2-(y-u)^2\right|
\le
|h|\,|y-u|+\frac{h^2}{2}
\le C_M .
\]
Exponentiating proves the claim.
\end{proof}

Define the classes of $k$-atomic mixing distributions and $k$-component Gaussian mixtures:
\[
\calM_{k,M}
\equiv
\{G\in\mathcal P([-M,M]):G\text{ has at most }k\text{ atoms}\},
\qquad
\calF_{k,M}\equiv\{f_G:G\in\calM_{k,M}\}.
\]

\begin{lemma}[Finite-mixture square-root envelope]
\label{lmm:pairwise-square-root-envelope}
There are constants $0<a_M\le 1/4$ and $A_M<\infty$ such that the
following holds.  
For any $k$ and any $f,g\in\calF_{k,M}$ such that
$\Hell(f,g)\leq \rho \le a_M$, we have \[
    |\sqrt{f(y)}-\sqrt{g(y)}|^2\le A_M [k(1+|y|)\rho^2
\wedge
\exp\{-((|y|-M)_+)^2/2\}] =: E_{k,\rho}(y)^2
\]
for all $y \in \R$, 
and the upper bound satisfies $\|E_{k,\rho}\|_{L_2(dx)}
\le
A_M\rho\sqrt{k\log(1/\rho)}.$

\end{lemma}

\begin{proof}
Fix $y\in\mathbb R$.  After combining equal locations, $f-g$ is a linear combination of $m\le2k$ shifted Gaussians.  Write
\[
(f-g)(x)=\varphi(x)E(x),
\qquad
E(x)=\sum_{j=1}^m a_j e^{u_jx-u_j^2/2}.
\]
Apply the key Lemma~\ref{lmm:nikolskii-exponential-sums} on the interval
$I_y$ from Lemma~\ref{lmm:local-gaussian-comparability}:
\begin{align*}
\left|\frac{(f-g)(y)}{\varphi(y)}\right|^2
&\le
\frac{C_M m}{|I_y|}
\int_{I_y}
\left|\frac{(f-g)(x)}{\varphi(x)}\right|^2\,dx \tag{\prettyref{lmm:nikolskii-exponential-sums}}
\\
&\le \frac{C_M m}{|I_y|\varphi(y)^2}
\int_{I_y}(f-g)(x)^2\,dx . \tag{\prettyref{lmm:local-gaussian-comparability}}
\end{align*}

Using $|I_y|^{-1}\le C_M(1+|y|)$ and $m\le2k$,
\[
|(f-g)(y)|^2
\le
C_M k(1+|y|)\int_{I_y}(f-g)(x)^2\,dx .
\]
Again by Lemma~\ref{lmm:local-gaussian-comparability},
$(f+g)(x)\le C_M(f+g)(y)$ on $I_y$.  Hence
\[
(\sqrt{f(y)}-\sqrt{g(y)})^2\leq 
\frac{|(f-g)(y)|^2}{(f+g)(y)}
\le
C_M k(1+|y|)
\int_{I_y}\frac{(f-g)(x)^2}{(f+g)(x)}\,dx
\le
2C_M k(1+|y|)\rho^2
\]
where we used $\int \frac{(f-g)^2}{f+g} \leq 2\Hell^2(f,g)$.
Finally, we have the tail bound
\[
(\sqrt{ f(y)}-\sqrt{g(y)})^2\leq 
f(y)+g(y)
\le
2\sup_{|u|\le M}\varphi(y-u)
\le
C_M\exp\{-((|y|-M)_+)^2/2\}.
\]
This shows $E_{k,\rho}$ is a valid upper bound for $|\sqrt f-\sqrt g|$.

It remains to integrate.  Let $R=A_M\sqrt{\log(1/\rho)}$, with $A_M$
large enough.  On $|y|\le R$,
\[
\int_{|y|\le R} E_{k,\rho}(y)^2\,dy
\le
A_M k\rho^2 R^2
\le
A_M k\rho^2\log(1/\rho).
\]
On $|y|>R$, integrating the Gaussian tail yields at most $\rho^4$. This shows the desired $\|E_{k,\rho}\|_2$.
\end{proof}

\begin{lemma}[Controlling finite-mixture approximation error]
\label{lmm:sqrt-quadrature-envelope}
For every fixed $M<\infty$ there is a constant $C_M<\infty$ such
that the following holds.  For every integer $k\ge1$ there is a nonnegative
function $D_k$, depending only on $k$ and $M$, such that
\[
\|D_k\|_{L_2(dx)}
\le
\left(\frac{C_M}{k}\right)^{k/4},
\]
and, for every probability measure $G$ supported on $[-M,M]$, there is a
probability measure $G_k$ supported on at most $k$ points in $[-M,M]$
for which
\[
\bigl|\sqrt{f_G(x)}-\sqrt{f_{G_k}(x)}\bigr|
\le D_k(x)
\qquad\text{for all }x\in\mathbb R .
\]
\end{lemma}

\begin{proof}
Let $G_k$ be a $k$-point Gauss quadrature for $G$, so that $G_k$
is supported on $[-M,M]$ and matches the first $2k-1$ moments of $G$;
see, for example, \cite[Theorem~3.3.3]{wu2020polynomial}.  Define
$q=f_G-f_{G_k}$.  Since the first $2k-1$ moments of $G-G_k$ vanish,
the Hermite expansion \prettyref{eq:hermite-exp} gives
\[
q(x)=\varphi(x)\sum_{m>2k-1}
\frac{\Delta_m}{\sqrt{m!}}h_m(x),
\qquad
\Delta_m=\int \theta^m\,(G-G_k)(d\theta)
\]
where 
$h_m=H_m/\sqrt{m!}$ denotes the Hermite polynomials orthonormal in
$L_2(\varphi)$.

By Cauchy--Schwarz, for any $x$,
\[
\frac{q(x)^2}{\varphi(x)}
\le
\left(\sum_{m>2k-1}\frac{\Delta_m^2e^{2m}}{m!}\right)
\left(\sum_{m>2k-1}e^{-2m}h_m(x)^2\varphi(x)\right).
\]
The second factor is uniformly bounded in $x$ by the diagonal Mehler formula
\cite[\href{https://dlmf.nist.gov/18.18.E28}{Eq.~(18.18.28)}]{NIST:DLMF}:
\[
\sum_{m\ge0}e^{-(2m+1)}h_m(x)^2\varphi(x)
=
(4\pi\sinh 2)^{-1/2}
\exp\!\left\{-\frac{x^2}{2}\tanh 1\right\}.
\]
For the first factor, since $|\Delta_m|\le2C_M^m$, applying Stirling's formula yields
\[
\sum_{m>2k-1}\frac{\Delta_m^2e^{2m}}{m!}
\le
\left(\frac{C_M}{k}\right)^k,
\]
with $C_M$ chosen large enough.  Therefore
$|f_G(x)-f_{G_k}(x)|\le (C_M/k)^{k/2} \sqrt{\varphi(x)}$
Finally, since 
$(\sqrt a-\sqrt b)^2\le |a-b|$ for $a,b\ge0$, we have 
 where
$|\sqrt{f_G(x)}-\sqrt{f_{G_k}(x)}| \leq (C_M/k)^{k/4}\varphi(x)^{1/4} \equiv D_k(x)$. 
\end{proof}

\begin{proof}[Proof of Lemma~\ref{lmm:scale-dependent-full-envelope}]
Fix $0<u\le c_M$ and let $L_u,k_u,s_u$ be as in the statement.  
After decreasing $c_M$, if
necessary, $\log(1/s_u)\le C_ML_u$ and
$\log\log(e^e/s_u)\ge c_M\log L_u$.  Choose the constant $A_M$ in the
definition of $k_u$ sufficiently large, depending only on $M$, 
so that
\[
k_u
\ge
C_M
\frac{\log(1/s_u)}{\log\log(e^e/s_u)}
\]
for sufficiently large $C_M$. 
After increasing $A_M$ once more, the preceding lower bound on $k_u$ implies $(C_M/k_u)^{k_u/4}\le s_u/16$. Hence
Lemma~\ref{lmm:sqrt-quadrature-envelope} gives $\|D_{k_u}\|_2\le s_u/16$.  Given
$f=f_G$ and $g=f_H$ with $\Hell(f,g)\leq s_u$, denote the corresponding $k_u$-point quadrature $f_k=f_{G_k}$ and $g_k=f_{H_k}$.  Then
\[
\Hell(f_k,g_k)
\le
\Hell(f,g)+2\|D_{k_u}\|_2
\le
2s_u .
\]
Since $s_u\le c_Mu\le c_M^2$, decrease $c_M$ further so that $2s_u\le a_M$ for all $0<u\le c_M$.
Let $E_{k_u,2s_u}$ be the finite-mixture envelope from
Lemma~\ref{lmm:pairwise-square-root-envelope}. Then
\[
|\sqrt f-\sqrt g|
\le
|\sqrt f-\sqrt {f_k}|
+|\sqrt {f_k}-\sqrt {g_k}|
+|\sqrt {g_k}-\sqrt g|
\le 2D_{k_u}+E_{k_u,2s_u} \equiv B_u .
\]
Moreover, since $\log(1/s_u)\le C_ML_u$,
\[
\|B_u\|_2
\le
2\|D_{k_u}\|_2
+C_Ms_u\sqrt{k_u\log(1/s_u)}
\le
s_u/8+C_Mc_Mu \leq u/4,
\]
upon choosing $c_M$ sufficiently small.
\end{proof}

\subsection{Proofs of Theorems \ref{thm:lcb-risk}--\ref{thm:npmle-risk} and Corollaries \ref{cor:lcb-regret}--\ref{cor:npmle-regret}}
\label{sec:remaining-proof}

\begin{proof}[Proof of \prettyref{thm:lcb-risk}]
Let $\tilde G$ denote the Birg\'e--Le Cam estimator based on pairwise comparison, whose Hellinger risk can be bounded using the local Hellinger entropy.
By \cite[Theorem~32.10]{polyanskiy.wu.2025}, the squared Hellinger risk of $\tilde G$ is bounded by a constant multiple of any
$\epsilon_n^2$ satisfying
$
n\epsilon_n^2
\ge
\log \coverN_{\mathrm{loc},H}(\calF_M,\epsilon_n)$.
The proof is completed by applying Theorem~\ref{thm:localcovering} above, which shows for all small
$\epsilon$,
\[
\log \coverN_{\mathrm{loc},H}(\calF_M,\epsilon)
\le
C_M\frac{\log(1/\epsilon)}{\log\log(1/\epsilon)}. \qedhere
\]
\end{proof}

\begin{proof}[Proof of \prettyref{thm:npmle-risk}]
This is an immediate consequence of combining the bracketing entropy integral in
\prettyref{thm:full-local-bracketing-integral} with the classical result of Wong and
Shen on the Hellinger risk of the MLE (see \cite[Theorem~2 and Remark~(ii) thereafter]{wong.shen.1995}).
\end{proof}

\begin{proof}[Proof of \prettyref{cor:lcb-regret}]
Let
\[
\psi(x)=x\frac{\log(1/\sqrt{x})}{\log\log(1/\sqrt{x})},
\qquad \psi(0)=0.
\]
For some sufficiently small absolute $x_0>0$, the function $\psi$ is
increasing and concave on $[0,x_0]$.   On $\{\hellinger^2(f_{\tilde G}, f_G)\le
x_0\}$,
\prettyref{thm:ub-compact-intro} gives
$\Regret(\tilde G\|G)\le C_M\psi(\hellinger^2(f_{\tilde G}, f_G))$.  Moreover,
$\Regret(\tilde G\|G)\le4M^2$, since both posterior means lie in
$[-M,M]$.  Hence, writing $\epsilon^2=\Expect_G \hellinger^2(f_{\tilde G}, f_G) $,
Jensen's and Markov's
inequalities give, for all sufficiently small $\epsilon^2$,
\[
\begin{aligned}
\Expect_G\Regret(\tilde G\|G)
&\le
C_M\Expect_G\!\left[\psi(\hellinger^2(f_{\tilde G}, f_G))\mathbf 1_{\{\hellinger^2(f_{\tilde G}, f_G)\le x_0\}}\right]
+4M^2\Prob_G(\hellinger^2(f_{\tilde G}, f_G)>x_0)\\
&\le C_M\psi(\epsilon^2)+\frac{4M^2}{x_0} \epsilon^2
\le C_M\psi(\epsilon^2).
\end{aligned}
\]
By \prettyref{thm:lcb-risk}, uniformly over
$G\in\mathcal P([-M,M])$,
$\epsilon^2\le C_M(\log n)/(n\log\log n)$.  Therefore
\[
\Expect_G\Regret(\tilde G\|G)
\le
\frac{C_M}{n}
\left(\frac{\log n}{\log\log n}\right)^2. 
\]
The above proves the upper bound.
The lower bound is from \cite{polyanskiy2021sharp}. \qedhere
\end{proof}

\begin{proof}[Proof of \prettyref{cor:npmle-regret}]
Without essential loss of generality we may assume $M'=M$.
Apply the same truncation argument in the proof of \prettyref{cor:lcb-regret}, now
using \prettyref{thm:npmle-risk} to bound
$\Expect_G\Hell^2(f_{\hat G_n},f_G)\le C_M(\log n)/n$.
\end{proof}

    \bibliographystyle{alpha}
    \bibliography{eb_regret_refs}

    \end{document}